\documentclass[a4paper]{amsart}

\usepackage[T1]{fontenc}
\usepackage[utf8]{inputenc}
\usepackage[english]{babel}
\usepackage{hyperref}
\usepackage{enumitem}
    \setlist[enumerate]{label = \upshape(\roman*)\leavevmode}
\usepackage{graphicx}
\usepackage{array}
\usepackage{todonotes}
\usepackage{xcolor}
    \definecolor{azure}{RGB}{0, 127, 255}
    \definecolor{green}{RGB}{0, 100, 0}
    \definecolor{blue}{RGB}{0, 0, 225}    \definecolor{orange}{RGB}{204, 85, 0}
\usepackage[backend=biber,sorting=nyt,giveninits=true,maxbibnames=9]{biblatex}

\usepackage{xpatch}

\makeatletter
\patchcmd\blx@bblinput{\blx@blxinit}
                      {\blx@blxinit
                      }{}{\fail}
\makeatother

\definecolor{darkred}{RGB}{139,0,0}
\definecolor{darkgreen}{RGB}{0,100,0}
\definecolor{darkmagenta}{RGB}{180,0,180}
\definecolor{darkblue}{RGB}{0,0,190}
\definecolor{darkorange}{RGB}{180,60,0}
\newcommand{\bsz}{{\bm{z}}}
\newcommand{\bsell}{{\bm{\ell}}}
\newcommand{\bsgamma}{{\bm{\gamma}}}
\newcommand{\setu}{{\mathfrak{u}}}
\newcommand{\calB}{{\mathcal{B}}}

\usepackage{mathtools, amssymb, amsthm}
\usepackage{cleveref, thmtools}  \usepackage{bm, dsfont} \numberwithin{equation}{section}

\newtheorem{theorem}{Theorem}[section]
\newtheorem{corollary}[theorem]{Corollary}
\newtheorem{lemma}[theorem]{Lemma}
\newtheorem{definition}[theorem]{Definition}
\newtheorem{remark}[theorem]{Remark}

\renewcommand{\hat}{\widehat}

\DeclareMathOperator{\Diag}{diag}
\DeclareMathOperator{\Span}{span}
\DeclareMathOperator{\Supp}{supp}
\DeclareMathOperator*{\Argmin}{arg\,min}
\newcommand{\ewc}{e^{\mathrm{wor-app}}}
\newcommand{\modone}{\bmod{}1}

\makeatletter
\newif\ifGPblacktext
\GPblacktexttrue
\providecommand{\colorrgb}[1]{}
\providecommand{\colorgray}[1]{}
\makeatother

\makeatletter
\begin{document}

\title[Minimal Subsampled Rank-1 Lattices]{Minimal Subsampled Rank-1 Lattices for Multivariate Approximation with Optimal Convergence Rate}

\author{Felix Bartel}
\address{Mathematisch-Geographische Fakultät, KU Eichstätt-Ingolstadt, 85270 Eichstätt, Germany.}
\email{felix.bartel@ku.de}

\author{Alexander~D.~Gilbert}
\address{School of Mathematics and Statistics, University of New South Wales, Sydney NSW 2052, Australia.}
\email{alexander.gilbert@unsw.edu.au}

\author{Frances~Y.~Kuo}
\address{School of Mathematics and Statistics, University of New South Wales, Sydney NSW 2052, Australia.}
\email{f.kuo@unsw.edu.au}

\author{Ian~H.~Sloan}
\address{School of Mathematics and Statistics, University of New South Wales, Sydney NSW 2052, Australia.}
\email{i.sloan@unsw.edu.au}


\begin{abstract} In this paper we show error bounds for randomly subsampled rank-1 lattices.
    We pay particular attention to the ratio of the size of the subset to the size of the initial lattice, which is decisive for the computational complexity.
    In the special case of Korobov spaces, we achieve the \emph{optimal polynomial sampling complexity} whilst having the \emph{smallest initial lattice possible}.
    We further characterize the frequency index set for which a given lattice is reconstructing by using the reciprocal of the worst-case error achieved using the lattice in question.
    This connects existing approaches used in proving error bounds for lattices.
    We make detailed comments on the implementation and test different algorithms using the subsampled lattice in numerical experiments.

\end{abstract} 

\maketitle

\section{Introduction} 

This paper mainly deals with the question of approximating functions using a subsample of a rank-1 lattice.
Rank-1 lattices have their origin as a Quasi-Monte Carlo (QMC) method and have been thoroughly studied in various settings ever since, see e.g., \cite{Sloan94, kaemmererdiss, KaPoVo13, PlPoStTa18, CKNS20, DKP22}.
The lattice point sets are given by $\bm X = \{\tfrac kn\bm z\modone\}_{k\in\{0,\dots,n-1\}}$, where $\bm z\in\mathds Z^d$ is known as the generating vector and it controls the quality of the point set, and $\bmod\,1$ denotes a component-wise fractional part of a vector.
The study of partial or subsampled rank-1 lattices has only begun recently, cf.\ \cite{BKPU22}.
Such point sets have the form
\begin{equation*}
    \bm X_J = \Big\{\tfrac kn\bm z \modone \Big\}_{k\in J}
    \quad\text{with}\quad
    J\subseteq\{0,\dots,n-1\} \,.
\end{equation*}
An example is depicted in \Cref{fig:scheme}.
The structure of lattice point sets allows for the application of fast Fourier methods when approximating functions whilst the subsampling improves the approximation properties, cf.\ \cite{BSU23}.
We investigate the worst-case setting, where the point set and algorithm are fixed and are used to approximate every function from a given function class $H$.
The \emph{worst-case error} for a general algorithm $A$ measured in $L_2$-norm is defined by
\begin{equation*}
    \ewc(A)
    \coloneqq \sup_{\|f\|_H \le 1} \|f-A f\|_{L_2} \,,
\end{equation*}
where $\|\cdot\|_H$ denotes the norm in $H$.
\begin{figure} \centering
    \begingroup
  \makeatletter
  \providecommand\color[2][]{\GenericError{(gnuplot) \space\space\space\@spaces}{Package color not loaded in conjunction with
      terminal option `colourtext'}{See the gnuplot documentation for explanation.}{Either use 'blacktext' in gnuplot or load the package
      color.sty in LaTeX.}\renewcommand\color[2][]{}}\providecommand\includegraphics[2][]{\GenericError{(gnuplot) \space\space\space\@spaces}{Package graphicx or graphics not loaded}{See the gnuplot documentation for explanation.}{The gnuplot epslatex terminal needs graphicx.sty or graphics.sty.}\renewcommand\includegraphics[2][]{}}\providecommand\rotatebox[2]{#2}\@ifundefined{ifGPcolor}{\newif\ifGPcolor
    \GPcolortrue
  }{}\@ifundefined{ifGPblacktext}{\newif\ifGPblacktext
    \GPblacktexttrue
  }{}\let\gplgaddtomacro\g@addto@macro
\gdef\gplbacktext{}\gdef\gplfronttext{}\makeatother
  \ifGPblacktext
\def\colorrgb#1{}\def\colorgray#1{}\else
\ifGPcolor
      \def\colorrgb#1{\color[rgb]{#1}}\def\colorgray#1{\color[gray]{#1}}\expandafter\def\csname LTw\endcsname{\color{white}}\expandafter\def\csname LTb\endcsname{\color{black}}\expandafter\def\csname LTa\endcsname{\color{black}}\expandafter\def\csname LT0\endcsname{\color[rgb]{1,0,0}}\expandafter\def\csname LT1\endcsname{\color[rgb]{0,1,0}}\expandafter\def\csname LT2\endcsname{\color[rgb]{0,0,1}}\expandafter\def\csname LT3\endcsname{\color[rgb]{1,0,1}}\expandafter\def\csname LT4\endcsname{\color[rgb]{0,1,1}}\expandafter\def\csname LT5\endcsname{\color[rgb]{1,1,0}}\expandafter\def\csname LT6\endcsname{\color[rgb]{0,0,0}}\expandafter\def\csname LT7\endcsname{\color[rgb]{1,0.3,0}}\expandafter\def\csname LT8\endcsname{\color[rgb]{0.5,0.5,0.5}}\else
\def\colorrgb#1{\color{black}}\def\colorgray#1{\color[gray]{#1}}\expandafter\def\csname LTw\endcsname{\color{white}}\expandafter\def\csname LTb\endcsname{\color{black}}\expandafter\def\csname LTa\endcsname{\color{black}}\expandafter\def\csname LT0\endcsname{\color{black}}\expandafter\def\csname LT1\endcsname{\color{black}}\expandafter\def\csname LT2\endcsname{\color{black}}\expandafter\def\csname LT3\endcsname{\color{black}}\expandafter\def\csname LT4\endcsname{\color{black}}\expandafter\def\csname LT5\endcsname{\color{black}}\expandafter\def\csname LT6\endcsname{\color{black}}\expandafter\def\csname LT7\endcsname{\color{black}}\expandafter\def\csname LT8\endcsname{\color{black}}\fi
  \fi
    \setlength{\unitlength}{0.0500bp}\ifx\gptboxheight\undefined \newlength{\gptboxheight}\newlength{\gptboxwidth}\newsavebox{\gptboxtext}\fi \setlength{\fboxrule}{0.5pt}\setlength{\fboxsep}{1pt}\definecolor{tbcol}{rgb}{1,1,1}\begin{picture}(5660.00,2260.00)\gplgaddtomacro\gplbacktext{\csname LTb\endcsname \put(1410,119){\makebox(0,0){\strut{}$\bm X$}}}\gplgaddtomacro\gplfronttext{}\gplgaddtomacro\gplbacktext{\csname LTb\endcsname \put(4230,119){\makebox(0,0){\strut{}$\bm X_J$}}}\gplgaddtomacro\gplfronttext{}\gplbacktext
    \put(0,0){\includegraphics[width={283.00bp},height={113.00bp}]{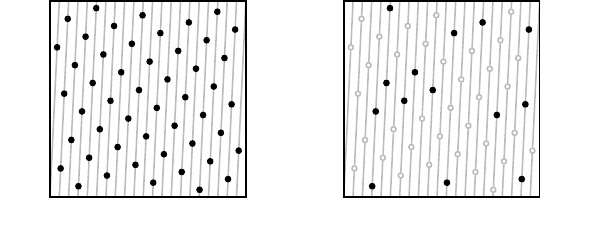}}\gplfronttext
  \end{picture}\endgroup
     \caption{Left: rank-1 lattice with generating vector $\bm z = (1,21)$ and lattice size $n = 55$. \\
    Right: subsampled rank-1 lattice with subsampling size $|J| = \lceil\sqrt n \, \log \sqrt{n}\,\rceil = 15$.}\label{fig:scheme}
\end{figure} 

For full lattices $\bm X = \bm X_{\{0,\dots,n-1\}}$ there are two main approaches to control the worst-case error:
\begin{enumerate}
\item
    Based on the function space $H$ and a lattice size $n$, one computes a generating vector $\bm z$ directly minimizing (a computable upper bound on) the worst-case error.
    It was shown in e.g., \cite{KSW06, CKNS20} that, for certain Korobov spaces with smoothness parameter $\alpha>1/2$, a component-by-component (CBC) construction achieves a decay rate
    \begin{equation}\label{eq:halfrate}
        \ewc(A_{\mathcal A}^{\bm X})
        \lesssim n^{-\alpha/2+\varepsilon} \,,
    \end{equation}
    where $\varepsilon>0$, $A_{\mathcal A}^{\bm X}$ is the ``classical lattice algorithm'' (see \eqref{eq:classical} below), and $A_n\lesssim B_n$ means there exists a constant $C>0$ independent of $n$ such that $A_n\le CB_n$ for all $n\in\mathds N$.
    When $\alpha$ is an integer it is the number of available square-integrable derivatives in each coordinate.
\item
    Another approach was taken in e.g., \cite{kaemmererdiss, KaPoVo13} where one first constructs a suitable frequency index set in the Fourier domain and then uses a CBC construction to end up with a ``reconstructing lattice'' (see precise definition in \eqref{eq:recon} below).
    This enables one to exactly reconstruct every trigonometric polynomial supported on the given frequencies.
    The ``aliasing'', i.e., the error arising from frequencies outside of the index set, is then controlled via an error bound.
    The resulting decay rate is the same as in \eqref{eq:halfrate}.
\end{enumerate}
In this paper we connect both approaches by finding a frequency index set for which a lattice is reconstructing in \Cref{recon} below.
These frequencies solely depend on the worst-case error achievable using any algorithm and the lattice in question.
In that way, lattices constructed according to the first approach can be investigated with respect to their reconstructing property, and tools from the second approach become available.

Regardless of the approach, there is a lower bound on the error rate using a full lattice of $n^{-\alpha/2}$, which is half the polynomial optimal rate of decay $n^{-\alpha}$ when there is no restriction for the point set used, cf.\ \cite{BKUV17}.
This seems to be an intrinsic problem with redundant information in the full lattice and motivates only using a subsampled lattice.
This was already investigated in the paper \cite{BKPU22}.

For unweighted Korobov spaces known as periodic Sobolev spaces with dominating mixed smoothness $\alpha>1/2$, a combination of random and deterministic subsampling was used in \cite[Corollary~1.2]{BKPU22}, yielding the bound
    \begin{equation}\label{eq:chemnitz}
        \ewc(S_{\mathcal B}^{\bm X_J})
        \lesssim |J|^{-\alpha} (\log|J|)^{(d-1)\alpha+\frac 12}
        \quad\text{with}\quad
        n \sim |J|^{\frac{2}{1-1/(2\alpha)}}
        \,,
    \end{equation}
    where $S_{\mathcal B}^{\bm X_J}$ is the ``least squares approximation'' to be defined precisely in \eqref{eq:lsqr} below, where $J$ labels the subsampled point set $\bm X_J$ and $A\sim B$ means $A\lesssim B$ and $B\lesssim A$ holds simultaneously.

The method in \cite{BKPU22} achieves the optimal polynomial order $|J|^{-\alpha+\varepsilon}$ up to $\varepsilon>0$ in terms of sampling complexity.
However, the computational complexity is determined by~$n$.
Since an algorithm using a subsampled lattice is inherently still using the lattice of size~$n$, the lower bound from \cite{BKUV17} still applies, and achieving the rate~$|J|^{-\alpha+\varepsilon}$ implies $n^{-\alpha/2} \lesssim |J|^{-\alpha+\varepsilon}$, yielding the lower bound
\begin{equation*}
    n \gtrsim |J|^{2-\frac{2\varepsilon}{\alpha}} \,.
\end{equation*}
With respect to the size of $|J|$, $n$ given in \eqref{eq:chemnitz} is larger still, being off by a polynomial order depending on the smoothness $\alpha$.

In this paper, in \Cref{koroboverror} and \Cref{cor} below, we show that we can go arbitrarily close to the squared number of points independently of the smoothness, i.e., we have
\begin{equation*}
    \ewc(S_{\mathcal B}^{\bm X_J})
    \lesssim |J|^{-\alpha+\varepsilon}
    \quad\text{with}\quad
    |J|^{2\sqrt{1-\varepsilon/\alpha}}
    \lesssim
    n
    \lesssim
    |J|^{2/\sqrt{1-\varepsilon/\alpha}}
\end{equation*}
for $0<\varepsilon<\alpha$.
This result is optimal in the sense that it achieves the best polynomial sampling complexity whilst having the smallest initial lattice size.

This paper is organized as follows.
We start by introducing the classical lattice algorithm and the kernel method in \Cref{sec:classical}.
Then we investigate the projection property of approximation algorithms in \Cref{sec:proj}, where we also prove a connection between the worst-case error and the reconstructing property of lattices in \Cref{recon}, which appears to be a new result.
In \Cref{sec:sub} we show general error bounds for the randomly subsampled lattice, which we apply to the Korobov space setting in \Cref{sec:korobov}.
After some comments on different implementations in \Cref{sec:implementation}, we present some numerical experiments in \Cref{sec:numerics} and give concluding remarks in \Cref{sec:conclusion}.

\vspace{5pt}
\paragraph{Notation.}
In this paper $\mathds T = \mathds R/\mathds N$ is the $1$-periodic torus;
$\langle\cdot,\cdot\rangle$ denotes the Euclidean inner product;
$\bm x\modone$ is understood as the component-wise fractional part of a vector $\bm x\in\mathds R^d$;
we write $a \equiv_n b$ if there exists $k\in\mathds N$ such that $a=b+kn$ and $a\not\equiv_n b$ if not;
$\|\cdot\|_{2}$ denotes the spectral norm;
vectors are understood as column vectors.
 \section{Classical lattice algorithm versus kernel method}\label{sec:classical}

A \emph{rank-1 lattice} point set $\bm X = \{\frac kn \bm z \modone \}_{k=0}^{n-1}$ is defined by a \emph{generating vector} $\bm z\in\mathds Z^d$ and the \emph{lattice size} $n\in\mathds N$.
Rank-1 lattices were first studied as a Quasi-Monte Carlo (QMC) method to compute integrals, see e.g., \cite{Sloan94}
\begin{equation*}
    \int_{\mathds T^d} f(\bm x) \;\mathrm d\bm x
    \approx \frac 1n \sum_{k=0}^{n-1} f(\tfrac kn \bm z \modone) \,,
\end{equation*}
which corresponds to approximating the Fourier coefficient $\hat f_{\bm 0}$ of a function with the Fourier series
\begin{equation*}
    f = \sum_{\bm h\in\mathds Z^d} \hat f_{\bm h} \exp(2\pi\mathrm i\langle\bm h, \cdot\rangle) \,,
    \quad\text{where}\quad
    \hat f_{\bm h} = \int_{\mathds T^d} f(\bm x) \exp(-2\pi\mathrm i\langle\bm h,\bm x\rangle) \;\mathrm d\bm x \,.
\end{equation*}
Doing this for multiple frequencies $\bm h$ from a prescribed frequency index set $\mathcal A\subseteq\mathds Z^d$, one may construct the trigonometric polynomial approximation to $f$ which we refer to as the \emph{classical lattice algorithm}
\begin{equation}\label{eq:classical}
    A_{\mathcal A}^{\bm X} f
    \coloneqq \sum_{\bm h\in\mathcal A}
    \Big( \underbrace{\frac 1n \sum_{k=0}^{n-1} f(\tfrac kn \bm z\modone) \exp(-2\pi\mathrm i\tfrac kn \langle\bm h,\bm z\rangle)}_{\eqqcolon\,\hat g_{\bm h}} \Big)
    \exp(2\pi\mathrm i\langle\bm h,\cdot\rangle) \,,
\end{equation}
see e.g., \cite{KSW06, KaPoVo13, kaemmererdiss, PlPoStTa18, CKNS20, DKP22}.
To get a feeling for this method, we have a look at the Fourier coefficients of the approximation $A_{\mathcal A}^{\bm X} f \eqqcolon g = \sum_{\bm h\in\mathcal A} \hat g_{\bm h} \exp(2\pi\mathrm i\langle\bm h,\cdot\rangle)$ in terms of the Fourier coefficients of $f$.
By the \emph{character property}
\begin{equation*}
    \frac 1n \sum_{k=0}^{n-1} \exp(2\pi\mathrm i m \tfrac kn) = \delta_{m\equiv_n 0}
    \coloneqq \begin{cases} 1 & \text{if } m\equiv_n 0, \\ 
    0 & \text{otherwise,} \end{cases}
\end{equation*}
we have for $\bm h\in\mathcal A$,
\begin{align}
    \hat g_{\bm h}
    &= \frac 1n \sum_{k=0}^{n-1} \Big( \sum_{\bm\ell\in\mathds Z^d} \hat f_{\bm\ell} \exp(2\pi\mathrm i\tfrac kn \langle\bm\ell,\bm z\rangle) \Big) \exp(-2\pi\mathrm i\tfrac kn\langle\bm h,\bm z\rangle) \nonumber \\
    &= \sum_{\bm\ell\in\mathds Z^d} \hat f_{\bm\ell} \, \frac 1n \sum_{k=0}^{n-1} \exp(2\pi\mathrm i\tfrac kn \langle\bm\ell-\bm h,\bm z\rangle) \nonumber \\
    &= \sum_{\substack{\bm\ell\in\mathds Z^d \\ \langle\bm\ell,\bm z\rangle\equiv_n\langle\bm h,\bm z\rangle}} \!\!\!\!\!\! \hat f_{\bm\ell} \label{eq:aliasing} \,.
\end{align}
Thus, the classical lattice approximation has an aliasing effect where multiple frequencies of the target function are mapped to a single frequency of the approximation.
In order to control this error we have to demand some prior smoothness assumptions on the function~$f$.
This is usually done by enforcing a decay in the Fourier coefficients as follows
\begin{equation}\label{eq:H}
    f\in H
    = \Big\{ f\colon\mathds T^d \to\mathds C : 
    \|f\|_H^2
    = \sum_{\bm h\in\mathds Z^d} r(\bm h) \, |\hat f_{\bm h}|^2
    < \infty \Big\} \,,
\end{equation}
where $r(\bm h)\to\infty$ for $\|\bm h\|_2 \to \infty$.
Specific choices of the function $r(\cdot)$ lead to e.g., the Korobov spaces in \Cref{sec:korobov}.
Obviously there are good and bad choices for the frequency index set $\mathcal A\subseteq\mathds Z^d$ and the generating vector $\bm z\in\mathds Z^d$.
The following theorem gives an error bound for the classical lattice algorithm and follows \cite[Section~2.4]{CKNS20}.

\begin{theorem}\label{classicalbound} Let $H$ be the function space defined in \eqref{eq:H}.
    Consider a lattice $\bm X = \{\tfrac kn\bm z\modone\}_{k=0}^{n-1} \subseteq\mathds T^d$ and a frequency index set
    $ \mathcal A = \{ \bm h\in\mathds Z^d : r(\bm h) \le M \} $ 
    for some radius $M>0$.
    Then for the classical lattice algorithm defined in \eqref{eq:classical}, we have
    \begin{equation}\label{eq:classicalboundupperbound}
        \ewc(A_{\mathcal A}^{\bm X})^2
        < \frac 1M + M \mathcal S_n(\bm z) \,,
    \end{equation}
    with
    \begin{equation}\label{eq:S}
        \mathcal S_n(\bm z)
        = \sum_{\bm h\in\mathds Z^d} \frac{1}{r(\bm h)} \sum_{\substack{\bm\ell\in\mathds Z^d\setminus\{\bm 0\}\\\langle\bm\ell,\bm z\rangle\equiv_n 0}} \frac{1}{r(\bm h+\bm\ell)} \,.
    \end{equation}
    In particular, with the optimal choice $M = 1/\sqrt{\mathcal S_n(\bm z)}$, which minimizes the right-hand side of \eqref{eq:classicalboundupperbound}, we have
    \begin{equation}\label{eq:classicalbound}
        \ewc(A_{\mathcal A}^{\bm X})^2
        <\frac 2M
        = 2 \sqrt{\mathcal S_n(\bm z)} \,.
    \end{equation}
\end{theorem} 

We have stated \eqref{eq:classicalboundupperbound} and \eqref{eq:classicalbound} as strict inequalities.
This can be seen from the derivation of the quantity $\mathcal S_n(\bm z)$, see e.g., \cite[Formula (2.7)]{CKNS20}, where infinitely many positive terms with index $\bm h\notin\mathcal A$ were added.

For weighted Korobov spaces (see \Cref{sec:korobov} below) and a given lattice size $n\in\mathds N$, component-by-component (CBC) constructions for generating vectors $\bm z$ are known, controlling $\mathcal S_n(\bm z)$ as in \Cref{Skorobov} below, which is \cite[Theorem~3.3]{KMN24}.
Fast CBC implementations are given in \cite{CKNS20, GS25}.

Next we comment on the kernel method, which is often used in practice and benefits from the lattice structure, cf.\ \cite[Section~5]{KKKNS21}.
For kernel methods we need to assume that $H$ is a reproducing kernel Hilbert space, i.e., function evaluations are continuous, which is a common and natural assumption.
By the Moore--Aronszajn theorem, $H$ has a corresponding reproducing kernel $K\colon\mathds T^d\times\mathds T^d\to\mathds C$, which is symmetric and positive definite.
The \emph{kernel approximation} of $f$ based on the lattice $\bm X$ is defined by
\begin{equation}\label{eq:kernelmethod}
    A_{\text{ker}}^{\bm X}f
    \coloneqq \sum_{k=0}^{n-1} a_k \, K(\cdot,\tfrac kn\bm z\modone)
    \quad\text{with}\quad
    \bm a = [a_k]_{k\in\{0,\dots, n-1\}} = \bm K^{-1}\bm f \,,
\end{equation}
where $\bm K = [K(\tfrac kn\bm z\modone, \tfrac{k'}{n}\bm z\modone)]_{k,k'\in \{0,\dots,n-1\}}$ is the kernel matrix, and the vector $\bm f = [f(\tfrac kn\bm z\modone)]_{k\in \{0,\dots,n-1\}}$ consists of the function evaluations on the lattice $\bm X$.
Since $K(\bm x,\bm y) = K(\bm x-\bm y,\bm 0)$, we have that $\bm K$ is a circulant matrix
\begin{align*}
    \bm K
    &= \Big[K(\tfrac kn\bm z\modone, \tfrac{k'}{n}\bm z\modone)\Big]_{k,k'\in \{0,\dots,n-1\}} \\
    &= \Big[K(\tfrac {k-k'}n\bm z\modone, \bm 0)\Big]_{k,k'\in \{0,\dots,n-1\}} \,.
\end{align*}
Thus, it is possible to diagonalize it by a discrete Fourier transform (DFT) matrix $\bm F = [\exp(2\pi\mathrm i \, k \, k'/n)]_{k,k'\in\{0,\dots,n-1\}}$ as
\begin{equation*}
    \bm K = \frac 1n \bm F \Diag( \bm F^\ast \bm c ) \bm F^\ast
    \quad\text{and}\quad
    \bm K^{-1} = \frac 1n \bm F \Diag( \bm F^\ast \bm c )^{-1} \bm F^\ast \,,
\end{equation*}
with $\bm c = [ K(\frac kn\bm z\modone,\bm 0) ]_{k\in\{0,\dots,n-1\}}$.
The kernel approximation with lattices is therefore computable using three fast Fourier transforms in $\mathcal O(n\log n)$ arithmetic operations.

In terms of the error it is known that for any given set of points, the kernel approximation \eqref{eq:kernelmethod} is optimal among all (linear or nonlinear) algorithms using function values at the same point set in the worst-case setting with respect to any error norm, cf.\ \cite{KKKNS21}.
In particular, we have that the worst-case error of the kernel approximation is bounded by the worst-case error of the classical lattice algorithm~\eqref{eq:classical}, i.e.,
\begin{equation*}
    \ewc(A_{\text{ker}}^{\bm X})
    \le \ewc(A_{\mathcal A}^{\bm X}) \,.
\end{equation*}
Thus, the bound from \Cref{classicalbound} applies to the kernel method as well.

 \section{Projection property of approximation algorithms}\label{sec:proj} 

A projection is a linear, idempotent map, meaning that the projection is equal to its composition with itself.
In this section we derive a condition under which the classical lattice algorithm $A_{\mathcal A}^{\bm X}$ defined in \eqref{eq:classical} is a projection.

Clearly we have linearity for $A_{\mathcal A}^{\bm X}$ but the idempotence is not always present.
For instance, if there are two aliasing frequencies, i.e., $\bm h,\bm h'\in\mathcal A$ with $\langle\bm h,\bm z\rangle \equiv_n \langle\bm h',\bm z\rangle$, we can take $f \coloneqq \exp(2\pi\mathrm i \langle \bm h, \cdot\rangle) + \exp(2\pi\mathrm i \langle \bm h', \cdot\rangle)$.
We then have $f\in \Span\{\exp(2\pi\mathrm i\langle\bm h,\cdot\rangle)\}_{\bm h\in\mathcal A}$ and yet by \eqref{eq:aliasing} we have
\begin{align*}
    A_{\mathcal A}^{\bm X} f
    = 2 f
    \neq f \,.
\end{align*}
Thus, to ensure idempotence it is necessary that there are no aliasing frequencies in the index set, that is, no indices have the same inner product with the generating vector $\bm z$ modulo $n$.

\begin{definition} A lattice $\bm X = \{\tfrac kn\bm z\modone\}_{k=0}^{n-1}\subseteq\mathds T^d$ is said to have the \emph{reconstructing property} on a frequency index set $\mathcal B\subseteq\mathds Z^d$ if 
    \begin{equation}\label{eq:recon}
        \langle\bm h,\bm z\rangle
        \not\equiv_n \langle\bm h',\bm z\rangle
        \quad\text{for all}\quad
        \bm h,\bm h'\in \mathcal B \,.
    \end{equation}
\end{definition} 

The reconstructing property is also sufficient for the idempotence as the following theorem shows.
For ease of readability, we will write $\mathcal B\subseteq\mathds Z^d$ when we have the reconstructing property and $\mathcal A\subseteq\mathds Z^d$ for a generic frequency index set.

\begin{theorem} Let $\mathcal B\subseteq\mathds Z^d$ be a nonempty frequency index set for which the lattice $\bm X = \{\frac kn\bm z\modone\}_{k=0}^{n-1} \subseteq\mathds T^d$ has the reconstructing property \eqref{eq:recon}.
    Then the classical lattice algorithm $A_{\mathcal B}^{\bm X}$ defined in \eqref{eq:classical} is a projection onto $\Span\{\exp(2\pi\mathrm i\langle\bm h,\cdot\rangle)\}_{\bm h\in\mathcal B}$.
\end{theorem} 

\begin{proof} By definition we have that $A_{\mathcal B}^{\bm X}$ is linear and maps onto the linear space $\Span\{\exp(2\pi\mathrm i\langle\bm h,\cdot\rangle)\}_{\bm h\in\mathcal B}$.
    Combining the aliasing formula \eqref{eq:aliasing} with the reconstructing property of $\bm X$ on $\mathcal B$ as in \eqref{eq:recon}, we obtain $\hat g_{\bm h} = \hat f_{\bm h}$ for all $\bm h\in\mathcal B$ and confirm the idempotence.
\end{proof} 

There are CBC constructions that compute a generating vector $\bm z$ for a given frequency index set $\mathcal B$ and $n\in\mathds N$ large enough that the reconstructing property holds, see e.g., \cite{kaemmerer14, kaemmerer20}.
We investigate the opposite question: given a lattice, what is the set of frequencies on which we have the reconstructing property, or equivalently what is the space on which the lattice algorithm is a projection.

\begin{theorem}\label{recon} Let $\bm X = \{\frac kn \bm z\modone\}_{k=0}^{n-1}\subseteq\mathds T^d$ be a lattice.
    Further, let $H$ be the function space from \eqref{eq:H} and $A^{\bm X}$ a (linear or non-linear) approximation algorithm using only samples from $\bm X$.
    Then $\bm X$ is a lattice with reconstructing property \eqref{eq:recon} for the frequency index set
    \begin{equation}\label{eq:hammerfall}
        \Big\{\bm h\in\mathds Z^d : r(\bm h) < (\ewc(A^{\bm X}))^{-2} \Big\} \,.
    \end{equation}
    In particular, the reconstructing property holds for the index set \eqref{eq:hammerfall} defined with the
    best possible worst-case error achievable by any (linear or non-linear) algorithm using points from the lattice $\bm X$.
\end{theorem} 

\begin{proof} The proof uses ideas from \cite[Theorem~3.2]{BKUV17}, which showed a lower bound on the worst-case error using samples from lattices in function spaces of dominating mixed smoothness.

    We will show that for any two aliasing frequencies, at least one must lie outside of the proposed frequency index set.

    Let $\bm h, \bm h'\in\mathds Z^d$ with $\bm h\neq\bm h'$, $\langle\bm h, \bm z\rangle \equiv_n \langle\bm h', \bm z\rangle$, and $r(\bm h)\le r(\bm h')$.
    We define the fooling function
    \begin{equation*}
        f \coloneqq \frac{\exp(2\pi\mathrm i\langle\bm h, \cdot\rangle) - \exp(2\pi\mathrm i\langle\bm h', \cdot\rangle) }{\sqrt{r(\bm h)+r(\bm h')}} \,,
    \end{equation*}
    so that
    \begin{equation*}
        \|f\|_{H}
        = \sqrt{\frac{r(\bm h)+r(\bm h')}{r(\bm h)+r(\bm h')}}
        = 1
        \quad\text{and}\quad
        \|f\|_{L_2} = \sqrt{\frac{2}{r(\bm h)+r(\bm h')}} \,.
    \end{equation*}
    Because of the aliasing frequencies, this function evaluates to zero on the lattice $\bm X$, i.e.,
    \begin{equation*}
        f(\tfrac kn\bm z\modone)
        = \frac{\exp(2\pi\mathrm i\frac kn \langle\bm h, \bm z\rangle) - \exp(2\pi\mathrm i\frac kn \langle\bm h', \bm z\rangle) }{\sqrt{r(\bm h)+r(\bm h')}}
        = 0
        \quad\text{for}\quad
        k\in\{0, \dots, n-1\}\,.
    \end{equation*}
    If $A^{\bm X}$ were to be linear, we would have $A^{\bm X}(f) \equiv 0$ and $\|f\|_{L_2} = \|f-A^{\bm X}(f)\|_{L_2} \le \ewc(A^{\bm X})$.
    Since we also include non-linear algorithms we only have $A^{\bm X}(f) = A^{\bm X}(-f)$ and need an intermediate step for the same result:
    \begin{align*}
        \|f\|_{L_2}
        &= \frac 12 \|f - A^{\bm X}(f) + f + A^{\bm X}(-f)\|_{L_2} \\
        &\le \frac 12 \Big(\|f - A^{\bm X}(f)\|_{L_2} + \|(-f) - A^{\bm X}(-f)\|_{L_2}\Big) \\
        &\le \max\Big\{ \|f - A^{\bm X}(f)\|_{L_2}, \|(-f) - A^{\bm X}(-f)\|_{L_2} \Big\} \\
        &\le \ewc(A^{\bm X}) \,.
    \end{align*}
    Thus
    \begin{equation*}
        \frac{1}{r(\bm h')}
        \le \frac{2}{r(\bm h)+r(\bm h')}
        = \|f\|_{L_2}^{2}
        \le \ewc(A^{\bm X})^2 \,.
    \end{equation*}
    Since this holds for any two aliasing frequencies, at least one of the two aliasing frequencies must lie outside of the proposed frequency index set \eqref{eq:hammerfall}, which is equivalent to the reconstructing property.
\end{proof} 

Applying this to the classical lattice algorithm, we obtain the following result.

\begin{corollary}\label{classicalrecon} Let $\bm X = \{\frac kn\bm z\modone\}_{k=0}^{n-1}\subseteq\mathds T^d$ be a lattice and $\mathcal S_n(\bm z)$ be defined as in \eqref{eq:S}.
    Then $\bm X$ has the reconstructing property on
    \begin{equation}\label{eq:majesty}
        \Big\{\bm h\in\mathds Z^d : r(\bm h) \le \frac{1}{2 \sqrt{\mathcal S_n(\bm z)}} \Big\} \,.
    \end{equation}
\end{corollary} 

\begin{proof} Using the classical lattice algorithm \eqref{eq:classical} in \Cref{recon}, we obtain the reconstructing property on
    \begin{equation*}
        \Big\{\bm h\in\mathds Z^d : r(\bm h) < (\ewc(A_{\mathcal A}^{\bm X}))^{-2}\Big\} \,.
    \end{equation*}
    Now we plug in the bound \eqref{eq:classicalbound}, which only makes the frequency index set smaller.
    By definition, $\bm X$ retains the reconstructing property on the smaller set.
    We have $\le$ instead of $<$ in \eqref{eq:majesty} since equality is not attained in \eqref{classicalbound}.
\end{proof} 

Note, the frequency index set \eqref{eq:majesty} has half of the radius of the frequency index set used in the classical lattice algorithm in \Cref{classicalbound}, where we have $M = 1/\sqrt{\mathcal S_n(\bm z)}$.
With the latter choice we do not have the guarantee for the reconstructing property.
This difference could be an artifact of the proof, however.

 \section{Error bounds for subsampled rank-1 lattices}\label{sec:sub} 

In this section we investigate what happens to the projection property as well as to error bounds if we use only a subset of the lattice points $\bm X_J = \{\frac kn\bm z\modone\}_{k\in J}$ for $J\subseteq\{0,\dots,n-1\}$.
Instead of using the classical lattice algorithm, we will use a \emph{least squares approximation}
\begin{equation}\label{eq:lsqr}
    S_{\mathcal B}^{\bm X_J} f
    \coloneqq \Argmin_{g\in \Span\{\exp(2\pi\mathrm i\langle\bm h,\cdot\rangle)\}_{\bm h\in\mathcal B}} \sum_{k\in J} \Big| g(\tfrac kn\bm z\modone) - f(\tfrac kn\bm z\modone) \Big|^2 \,.
\end{equation}
Note that it is necessary for $\bm X$ to have the reconstructing property on $\mathcal B$ as otherwise there would be functions with aliasing frequencies having the same function values along the lattice and subsequently also along the subsampled lattice.
This would result in the non-uniqueness of the least squares approximation.
Furthermore, bad choices in $J$ would not work as well.
For instance, we need at least $|J|\ge|\mathcal B|$, otherwise we could construct a non-zero function that evaluates to zero on $\bm X_J$, again, yielding non-uniqueness.

Throughout this paper we will use multisets for $J$, meaning that elements may repeat.
For the cardinality $|J|$, elements are counted according to their repeated appearance, e.g.\ $|\{1, 1\}| = 2$.
Note that for the actual number of points used, or equivalently, the number of function evaluations used, it is not necessary to count by repeated appearance but doing so always provides an upper bound.

\begin{lemma}\label{lsqrcoeffs} Let $\mathcal B\subseteq\mathds Z^d$ be a nonempty frequency index set for which the lattice $\bm X = \{\frac kn\bm z\modone\}_{k=0}^{n-1} \subseteq\mathds T^d$ has the reconstructing property.
    For $J\subseteq\{0,\dots,n-1\}$ a multiset with $|J|\ge|\mathcal B|$, suppose that
    \begin{equation}\label{eq:L}
        \bm L_{J,\mathcal B}
        \coloneqq \Big[ \exp(2\pi\mathrm i\tfrac kn\langle\bm h, \bm z\rangle) \Big]_{k\in J,\, \bm h\in\mathcal B}
        \in\mathds C^{|J|\times|\mathcal B|}
    \end{equation}
    has full rank $|\mathcal B|$, and let $\bm f_{J} \coloneqq [f(\frac kn\bm z\modone)]_{k\in J}$ for $f\colon\mathds T^d\to\mathds C$.
    Then for the least squares approximation 
    $S_{\mathcal B}^{\bm X_J} f = \sum_{\bm h\in\mathcal B}\hat g_{\bm h} \exp(2\pi\mathrm i\langle\bm h,\cdot\rangle)$
    defined in \eqref{eq:lsqr}, we have the Fourier coefficients
    \begin{equation*}
        \bm{\hat g}_{\mathcal B}
        \coloneqq [\hat g_{\bm h}]_{\bm h\in\mathcal B}
        = (\bm L_{J,\mathcal B}^{\ast} \, \bm L_{J,\mathcal B})^{-1} \bm L_{J,\mathcal B}^{\ast} \, \bm f_J \,.
    \end{equation*}
\end{lemma} 

\begin{proof} The definition of the least squares approximation \eqref{eq:lsqr} can be rewritten in matrix-vector form using the wanted Fourier coefficients
    \begin{equation*}
        \Argmin_{\bm{\hat g}_{\mathcal B} \in\mathds C^{|\mathcal B|}}
        \| \bm L_{J,\mathcal B} \, \bm{\hat g}_{\mathcal B} - \bm f_J \|_{2}^{2} \,.
    \end{equation*}
    Computing this minimum, one uses the normal equation
    \begin{equation*}
        \bm L_{J,\mathcal B}^\ast \, \bm L_{J,\mathcal B} \, \bm{\hat g}_{\mathcal B} 
        = \bm L_{J,\mathcal B}^\ast \, \bm f_J \,.
    \end{equation*}
    By assumption the left-hand side matrix $\bm L_{J,\mathcal B}^\ast \, \bm L_{J,\mathcal B}$ has full rank and can therefore be inverted, yielding the assertion.
\end{proof} 

Note, that the full rank condition already implies that $\bm X$ has to have the reconstructing property on $\mathcal B$ and $|J|\ge|\mathcal B|$.
Next we show that given full rank of $\bm L_{J,\mathcal B}$, the least squares approximation still maintains the projection property.

\begin{lemma}\label{lsqrprojection} Let $\mathcal B\subseteq\mathds Z^d$ be a nonempty frequency index set for which the lattice $\bm X = \{\frac kn\bm z\modone\}_{k=0}^{n-1} \subseteq\mathds T^d$ has the reconstructing property.
    For $J\subseteq\{0,\dots,n-1\}$ a multiset with $|J|\ge|\mathcal B|$, suppose $\bm L_{J,\mathcal B}$ from \eqref{eq:L} has full rank $|\mathcal B|$.
    Then the least squares approximation $S_{\mathcal B}^{\bm X_J}$ defined in \eqref{eq:lsqr}, with the subsampled lattice $\bm X_J$, is a projection onto $\Span\{\exp(2\pi\mathrm i\langle\bm h,\cdot\rangle)\}_{\bm h\in\mathcal B}$.

\end{lemma} 

\begin{proof} By definition we have that $S_{\mathcal B}^{\bm X_J}$ maps linearly into $\Span\{\exp(2\pi\mathrm i\langle\bm h,\cdot\rangle)\}_{\bm h\in\mathcal B}$.
    To show the idempotence, we use the Fourier representation of a function supported on $\mathcal B$, i.e.,
    \begin{equation*}
        f = \sum_{\bm h\in\mathcal B} \hat f_{\bm h}\exp(2\pi\mathrm i\langle\bm h,\cdot\rangle) \in\Span\{\exp(2\pi\mathrm i\langle\bm h,\cdot\rangle)\}_{\bm h\in\mathcal B}\,,
    \end{equation*}
    and collect the Fourier coefficients in the vector $\bm{\hat f}_{\mathcal B} = [\hat f_{\bm h}]_{\bm h\in\mathcal B}$.
    From the reconstructing property, there holds $\bm f_J = \bm L_{J,\mathcal B}\,\bm{\hat f}_{\mathcal B}$.
    By \Cref{lsqrcoeffs} we have for the Fourier coefficients of the least squares approximation
    \begin{equation*}
        \bm{\hat g}_{\mathcal B}
        = (\bm L_{J,\mathcal B}^\ast\,\bm L_{J,\mathcal B})^{-1} \bm L_{J,\mathcal B}^\ast \, \bm f_J
        = (\bm L_{J,\mathcal B}^\ast\,\bm L_{J,\mathcal B})^{-1} \bm L_{J,\mathcal B}^\ast \, \bm L_{J,\mathcal B} \, \bm{\hat f}_{\mathcal B}
        = \bm{\hat f}_{\mathcal B} \,,
    \end{equation*}
    where the inverse exists because of the assumptions on $\bm L_{J,\mathcal B}$ having full rank.
    The Fourier coefficients of the least squares approximation coincide with those of the function~$f$.
    Thus, they are the same functions, which proves the assertion.
\end{proof} 

Before commenting on a specific choice of $J$, we relate the worst-case error of the least squares approximation to the spectral properties of certain matrices.

\begin{theorem}\label{generalbound} Let $\mathcal B\subseteq\mathds Z^d$ be a nonempty frequency index set for which the lattice $\bm X = \{\frac kn\bm z\modone\}_{k=0}^{n-1} \subseteq\mathds T^d$ has the reconstructing property.
    For $J\subseteq\{0,\dots,n-1\}$ a multiset with $|J|\ge|\mathcal B|$, suppose $\bm L_{J,\mathcal B}$ in \eqref{eq:L} has full rank, and define
    \begin{equation}\label{eq:Phi}
        \bm\Phi_{J,\mathcal B}
        \coloneqq \Big[
        \frac{\exp(2\pi\mathrm i\frac{k}{n}\langle\bm h,\bm z\rangle)}{\sqrt{r(\bm h)}}
        \Big]_{k\in J,\,\bm h\notin\mathcal B} \,.
    \end{equation}
    Then the worst-case error of the least squares approximation defined in \eqref{eq:lsqr} with respect to the function space defined in \eqref{eq:H} satisfies the bound
    \begin{equation*}
        \ewc(S_{\mathcal B}^{\bm X_J})^2
        \le \sup_{\bm h\notin\mathcal B} \frac{1}{r(\bm h)}
        + 
        \frac{ \|\bm\Phi_{J,\mathcal B}\|_{2}^{2} }{ \sigma_{\min}^2(\bm L_{J,\mathcal B}) } \,,
    \end{equation*}
    where $\sigma_{\min}(\bm L_{J, \mathcal B})$ denotes the smallest singular value of $\bm L_{J,\mathcal B}$.
\end{theorem} 

\begin{proof} For $f\in H$ we write $f-S_{\mathcal B}^{\bm X_J}f = (f-P_{\mathcal B}f) + (P_{\mathcal B}f-S_{\mathcal B}^{\bm X_J}f)$, where $P_{\mathcal B} f = \sum_{\bm h\in\mathcal B} \hat f_{\bm h} \exp(2\pi\mathrm i\langle\bm h,\cdot\rangle)$ is the $L_2$-projection of $f$ onto $\Span\{\exp(2\pi\mathrm i\langle\bm h,\cdot\rangle)\}_{\bm h\in\mathcal B}$.
    This decomposition is orthogonal because by definition $S_{\mathcal B}^{\bm X_J}f = P_{\mathcal B} \, S_{\mathcal B}^{\bm X_J}f$.
    Thus,
    \begin{equation}\label{eq:errordecomposition}
        \|S_{\mathcal B}^{\bm X_J} f - f\|_{L_2}^2
        =
        \|f - P_{\mathcal B} f\|_{L_2}^2
        + \|P_{\mathcal B} f - S_{\mathcal B}^{\bm X_J} f\|_{L_2}^2 \,.
    \end{equation}
    For the first summand of \eqref{eq:errordecomposition} we use
    \begin{equation*}
        \|f - P_{\mathcal B} f\|_{L_2}^2
        = \sum_{\bm h\notin\mathcal B} \frac{1}{r(\bm h)} \, r(\bm h) \, |\hat f_{\bm h}|^2
        \le \sup_{\bm h\notin\mathcal B} \frac{1}{r(\bm h)} \, \|f\|_H^2 \,.
    \end{equation*}

    Now we estimate the second summand of \eqref{eq:errordecomposition}.
    By \Cref{lsqrprojection}, $S_{\mathcal B}^{\bm X_J}$ is a projection onto $\Span\{\exp(2\pi\mathrm i\langle\bm h,\cdot\rangle)\}_{\bm h\in\mathcal B}$ and we obtain using \Cref{lsqrcoeffs}
    \begin{align*}
        \|P_{\mathcal B} f - S_{\mathcal B}^{\bm X_J} f\|_{L_2}^2
        &= \|S_{\mathcal B}^{\bm X_J} P_{\mathcal B} f - S_{\mathcal B}^{\bm X_J} f\|_{L_2}^2
        = \|S_{\mathcal B}^{\bm X_J} (P_{\mathcal B} f - f)\|_{L_2}^2 \\
        &\le \Big\|(\bm L_{J,\mathcal B}^\ast \, \bm L_{J,\mathcal B})^{-1} \bm L_{J,\mathcal B}^{\ast} \Big[(P_{\mathcal B} f-f) (\tfrac kn\bm z \modone)\Big]_{k\in J} \Big\|_{2}^2 \\
        &\le \|(\bm L_{J,\mathcal B}^\ast \, \bm L_{J,\mathcal B})^{-1} \bm L_{J,\mathcal B}^{\ast} \|_{2}^2 \, \sum_{k\in J} \Big|(P_{\mathcal B} f-f) (\tfrac kn\bm z \modone) \Big|^2 \,.
    \end{align*}
    By \cite[Proposition~3.1]{KUV19}, we have $\| (\bm L_{J,\mathcal B}^\ast \, \bm L_{J,\mathcal B})^{-1}\bm L_{J,\mathcal B}^\ast \|_{2} = 1/\sigma_{\min} (\bm L_{J,\mathcal B})$.
    On the other hand, we have
    \begin{align*}
        \sum_{k\in J} \Big|(f-P_{\mathcal B} f) (\tfrac kn \bm z \modone) \Big|^2
        &= \sum_{k\in J} \Big| \sum_{\bm h\notin\mathcal B} \hat f_{\bm h} \exp(2\pi\mathrm i \tfrac{k}{n}\langle\bm h, \bm z\rangle) \Big|^2 \\
        &= \sum_{k\in J} \Big| \sum_{\bm h\notin\mathcal B} \sqrt{r(\bm h)} \, \hat f_{\bm h} \, \frac{1}{\sqrt{r(\bm h)}} \, \exp(2\pi\mathrm i \tfrac{k}{n}\langle\bm h, \bm z\rangle) \Big|^2 \\
        &= \|\bm\Phi_{J,\mathcal B} \, [\sqrt{r(\bm h)} \, \hat f_{\bm h}]_{\bm h\notin\mathcal B} \|_{2}^{2} \\
        &\le \|\bm\Phi_{J,\mathcal B}\|_{2}^{2} \, \|f\|_H^2 \,.
    \end{align*}
    Combining the two estimates completes the proof
\end{proof} 

We have seen in \Cref{lsqrprojection} that it is possible to reconstruct trigonometric polynomials from values at a subsampled lattice $\bm X_J = \{\frac kn \bm z\modone\}_{k\in J} \subseteq\mathds T^d$ for $J\subseteq\{0, \dots, n-1\}$ given that $\bm L_{J,\mathcal B}$ has full rank.
The following lemma quantifies this for a randomly selected subset $J$.

\begin{lemma}\label{constructions} Let $\mathcal B\subseteq\mathds Z^d$ be a nonempty frequency index set for which the lattice $\bm X = \{\frac kn\bm z\modone\}_{k=0}^{n-1} \subseteq\mathds T^d$ has the reconstructing property.
    Further, let $t>0$ and $J\subseteq\{0, \dots, n-1\}$ be a multiset of uniformly i.i.d.\ drawn integers with
    \begin{equation}\label{eq:12}
        |J| \ge 12 \, |\mathcal B| \, (\log|\mathcal B| + t) \,.
    \end{equation}
    For $\bm L_{J,\mathcal B}$ defined as in \eqref{eq:L} we have the following two inequalities, each holding with probability exceeding $1-\exp(-t)$,
    \begin{equation}\label{eq:sepultura}
        \frac{|J|}{2} \le \sigma_{\min}^2(\bm L_{J,\mathcal B})
        \quad\text{and}\quad
        \sigma_{\max}^2(\bm L_{J,\mathcal B}) \le \frac{3|J|}{2} \,.
    \end{equation}
    For $|J|\ge 3$, $\bm\Phi_{J,\mathcal B}$ as in \eqref{eq:Phi}, and $\bm\Phi_{\mathcal B} = \bm\Phi_{\{0,\dots,n-1\},\mathcal B}$, we have with probability exceeding $1-2^{3/4}\exp(-t)$ that
    \begin{equation}\label{eq:watain}
        \|\bm\Phi_{J,\mathcal B}\|_{2}^2
        \le 21 \, (\log|J|+t) \sum_{\bm h\notin\mathcal B} \frac{1}{r(\bm h)} + 2\frac{|J|}{n} \|\bm\Phi_{\mathcal B}\|_{2}^2 \,.
    \end{equation}
\end{lemma} 

\begin{proof} By \cite[Lemma~2.2]{BKPU22} the reconstructing property of $\bm X$ corresponds to an exact $L_2$-Marcinkiewicz--Zygmund inequality, i.e.,
    \begin{equation*}
        \|f\|_{L_2}^{2}
        = \frac 1n \sum_{k=0}^{n-1} |f(\tfrac kn \bm z\modone)|^2
        \quad\text{for all}\quad
        f\in\Span\{\exp(2\pi\mathrm i\langle\bm h,\cdot\rangle)\}_{\bm h\in\mathcal B} \,.
    \end{equation*}
    With the logarithmic oversampling assumption \eqref{eq:12} on $J$, we are able to use the technique of \cite[Theorem~3.1]{BKPU22}, which yields that the subsampled points fulfill an $L_2$-Marcinkiewicz--Zygmund inequality with constants $1/2$ and $3/2$, i.e.,
    \begin{equation*}
        \frac{1}{2} \|f\|_{L_2}^{2}
        \le \frac{1}{|J|} \sum_{k\in J} |f(\tfrac kn \bm z\modone)|^2
        \le \frac{3}{2} \|f\|_{L_2}^{2}
        \;\;\text{for all}\;\;
        f\in\Span\{\exp(2\pi\mathrm i\langle\bm h,\cdot\rangle)\}_{\bm h\in\mathcal B} \,,
    \end{equation*}
    where from their proof each bound holds with probability $1-\exp(-t)$, respectively.
    This random technique uses a matrix Chernoff concentration inequality \cite[Theorem~1.1]{Tropp11} with one of the first applications in the sampling context appearing in \cite[Theorem~2.1]{CM17}.
    Using \cite[Lemma~2.2]{BKPU22} again converts this to the desired bounds on the singular values in \eqref{eq:sepultura}.

    For $k\in\{0, \dots, n-1\}$, we define the sequences
    \begin{equation*}
        \bm y^k \coloneqq \Big[ \frac{\exp(-2\pi\mathrm i\frac kn \langle\bm h, \bm z\rangle)}{\sqrt{r(\bm h)}} \Big]_{\bm h\notin\mathcal B} \,,
    \end{equation*}
    with $\|\bm y^k\|_2^2 \le \sum_{\bm h\notin\mathcal B} 1/r(\bm h)$, so that
    \begin{equation*}
        \frac{1}{|J|} \sum_{k\in J} \bm y^k(\bm y^k)^\ast
        = \frac{1}{|J|} \sum_{k\in J}\Big[ \frac{\exp(2\pi\mathrm i\tfrac kn\langle\bm h'-\bm h, \bm z\rangle)}{\sqrt{r(\bm h)r(\bm h')}}\Big]_{\bm h,\bm h'\notin\mathcal B}
        = \frac{1}{|J|} \bm\Phi_{J,\mathcal B}^{\ast}\bm\Phi_{J,\mathcal B} \,.
    \end{equation*}
    The set $J$ gives us a random set of sequences $\{\bm y^k\}_{k\in J}$ for which we know
    \begin{equation*}
        \mathds E_k(\bm y^k (\bm y^k)^\ast)
        = \frac 1n \sum_{k=0}^{n-1} \Big[ \frac{\exp(2\pi\mathrm i\frac kn \langle\bm h'-\bm h, \bm z\rangle)}{\sqrt{r(\bm h) \, r(\bm h')}} \Big]_{\bm h,\bm h'\notin\mathcal B}
        = \frac 1n \bm\Phi_{\mathcal B}^\ast \, \bm\Phi_{\mathcal B} \,,
    \end{equation*}
    where the expectation is taken with respect to the discrete random variable $k$ uniformly sampled from $\{0,\dots,n-1\}$.
    Applying \cite[Proposition~3.8]{MU21} on these random sequences (with $n$ there replaced by $|J|\ge 3$) yields with probability exceeding $1-2^{3/4} |J|^{1-\beta}$ for $\beta>1$ that
    \begin{equation*}
        \Big\| \frac{1}{|J|} \bm\Phi_{J,\mathcal B}^{\ast}\bm\Phi_{J,\mathcal B} - \frac 1n \bm\Phi_{\mathcal B}^\ast \, \bm\Phi_{\mathcal B} \Big\|_{2}
        \le \max\Big\{\frac{21 \beta\log |J|}{|J|} \sum_{\bm h\notin\mathcal B} \frac{1}{r(\bm h)},\, \frac 1n \|\bm\Phi_{\mathcal B}^\ast \, \bm\Phi_{\mathcal B}\|_{2}\Big\} \,.
    \end{equation*}
    Setting $\beta = (\log |J|+t)/\log |J|$ for $t>0$, this gives with probability exceeding $1-2^{3/4}\exp(-t)$ that
    \begin{equation*}
        \Big\| \frac{1}{|J|} \bm\Phi_{J,\mathcal B}^{\ast}\bm\Phi_{J,\mathcal B} - \frac 1n \bm\Phi_{\mathcal B}^\ast \, \bm\Phi_{\mathcal B} \Big\|_{2}
        \le \max\Big\{\frac{21 \, (\log |J|+t)}{|J|} \sum_{\bm h\notin\mathcal B} \frac{1}{r(\bm h)},\, \frac 1n \|\bm\Phi_{\mathcal B}^\ast\,\bm\Phi_{\mathcal B}\|_{2}\Big\} \,,
    \end{equation*}
    from which the assertion \eqref{eq:watain} follows by basic arithmetic operations.
\end{proof} 

Using these spectral bounds in the general error bound for least squares approximation yields the following theorem.

\begin{theorem}\label{withPhi} Let $\mathcal B\subseteq\mathds Z^d$ be a nonempty frequency index set for which the lattice $\bm X = \{\frac kn\bm z\modone\}_{k=0}^{n-1} \subseteq\mathds T^d$ has the reconstructing property.
    Further, let $t\ge4$ and $J\subseteq\{0, \dots, n-1\}$ be a multiset of uniformly i.i.d.\ drawn integers with
    \begin{equation*}
|J| := \lceil 12 \, |\mathcal B| \, (\log|\mathcal B| + t)\rceil\,.
    \end{equation*}
    Then the least squares approximation defined in \eqref{eq:lsqr} using samples on the subsampled lattice $\bm X_J$ satisfies with probability exceeding $1-3\exp(-t)$ that
    \begin{equation*}
        \ewc(S_{\mathcal B}^{\bm X_J})^2
        \le \sup_{\bm h\notin\mathcal B} \frac{1}{r(\bm h)}
        + 
        \frac{42\,(\log|J|+t)}{|J|}
\sum_{\bm h\notin\mathcal B} \frac{1}{r(\bm h)} + \frac{4}{n} \|
        \bm\Phi_{\mathcal B}\|_{2}^2 \,,
    \end{equation*}
    with respect to the function space defined in \eqref{eq:H} and $\bm\Phi_{\mathcal B} = \bm\Phi_{\{0,\dots,n-1\},\mathcal B}$ as in \eqref{eq:Phi}. Additionally we have
    \begin{equation}\label{eq:JB}
         \frac{6\,(\log|J|+t)}{|J|} 
         \le \frac{1}{|\mathcal B|} 
         \le \frac{13\,(\log|J|+t)}{|J|} \,.
    \end{equation}
\end{theorem} 

\begin{proof} We use the bound of the smallest singular value of $\bm L_{J,\mathcal B}$ and the bound on the spectral norm of $\bm\Phi_{J,\mathcal B}$ of \Cref{constructions}.
    By Boole's inequality, the overall probability exceeds $1-\exp(-t)-2^{3/4}\exp(-t) \ge 1-3\exp(-t)$, which is the claimed probability.
    Plugging in the spectral bounds into \Cref{generalbound}, we obtain
    \begin{equation*}
        \ewc(S_{\mathcal B}^{\bm X_J})^2
        \le \sup_{\bm h\notin\mathcal B} \frac{1}{r(\bm h)}
        + 
        \frac{2}{|J|} \Big( 21 \,(\log|J|+t) \sum_{\bm h\notin\mathcal B} \frac{1}{r(\bm h)} + 2\frac{|J|}{n} \|\bm\Phi_{\mathcal B}\|_{2}^2\Big) \,.
    \end{equation*}
    It remains to estimate the scaling factor for the sum.

    We first show that for $b\ge 1$ and $t\ge 4$ we have
    \begin{equation}\label{eq:hard}
        12\,b\,(\log b+t)
        \le b^2\exp(t) \,.
    \end{equation}
    Since $\log b \le b-1$ this follows from the stronger assertion
    \begin{equation*}
        12\,(b-1+t)
        \le b\exp(t)
        \quad\Leftrightarrow\quad
        (\exp(t)-12)\,b +12-12t
        \ge 0 \,.
    \end{equation*}
    The left-hand side is non-decreasing in $b$ since $t\ge 4 \ge \log(12)$.
    With $b\ge 1$, it remains to check the non-negativity for $b=1$, which is satisfied for $t\ge 4$ showing~\eqref{eq:hard}.

    We are ready to prove an upper bound on $(\log j+t)/j$ in terms of $1/b$ when $j := \lceil\,12\,b\,(\log b+t)\,\rceil \ge 12\,b\,(\log b+t)$.
    Since $(\log j+t)/j$ is increasing with decreasing $j$, it suffices to show an estimate for $\ell\le j$ such that
    \begin{equation*}
        12\,b\,(\log b+t)
        \le \ell
        \le b^2\exp(t) \,.
    \end{equation*}
    Such a value of $\ell\in\mathds R$ exists due to \eqref{eq:hard}.
    Then it follows that
    \begin{equation*}
        \frac{\log j+t}{j}
        \le \frac{\log\ell+t}{\ell}
        \le \frac{2 \, (\log b + t)}{12 \, b \, (\log b + t)}
        = \frac{1}{6\,b} \,.
    \end{equation*}
    
    Finally we prove an upper bound on $1/b$ in terms of $(\log j+t)/j$.
    Since $b\ge 1$ and $t\ge 4$, we have $12 \, b \, (\log b+t) \ge \frac{48}{49} \lceil 12 \, b \, (\log b+t)\rceil$.
    Thus
    \begin{equation*}
        \frac{1}{b}
        \le \frac{49}{48} \frac{12 \, (\log b+t)}{\lceil 12 \, b \, (\log b+t)\rceil}
        \le \frac{13 \, (\log j+t)}{j} \,,
    \end{equation*}
    where we used $j > 12\,b \ge b$ for $\log b < \log j$ and $49\cdot 12/48 <13$.
    This completes the proof.
\end{proof} 

In \Cref{withPhi} the spectral norm of an infinite-dimensional matrix $\|\bm\Phi_{\mathcal B}\|_{2}^{2}$ occurs.
We have
\begin{equation*}
    \|\bm\Phi_{\mathcal B}\|_{2}^{2} =\| \bm\Phi_{\mathcal B}^\ast \, \bm\Phi_{\mathcal B} \|_{2} = \| \bm\Phi_{\mathcal B} \, \bm\Phi_{\mathcal B}^\ast \|_{2} \,,
\end{equation*}
where the first matrix product is infinite-dimensional, whereas the second matrix-product is finite-dimensional (and may be used for computing the spectral norm numerically).
In the following lemma we use the former to show that this quantity can be upper bounded in terms of the known quantity $\mathcal S_n(\bm z)$ from \eqref{eq:S}.
Note that the reconstructing property is not needed in this lemma.

\begin{lemma}\label{nokings} Let $\mathcal B\subseteq\mathds Z^d$ be a frequency index set and $\bm X = \{\frac kn \bm z\modone\}_{k=0}^{n-1}\subseteq\mathds T^d$ a lattice.
    Further let $\bm\Phi_{\mathcal B} \coloneqq \bm\Phi_{\{0,\dots,n-1\}, \mathcal B}$ be as in \eqref{eq:Phi} and $\mathcal S_n(\bm z)$ as in \eqref{eq:S}.
    Then
    \begin{equation*}
        \frac 1n \|\bm\Phi_{\mathcal B}\|_{2}^2
        \le \sup_{\bm h\notin\mathcal B} \frac{1}{r(\bm h)} + \sqrt{\mathcal S_n(\bm z)} \,.
    \end{equation*}
\end{lemma} 

\begin{proof} By the definition of $\bm\Phi_{\mathcal B}$, we have
    \begin{equation*}
        \frac 1n \|\bm\Phi_{\mathcal B}\|_{2}^2
        = \Big\| \frac 1n \bm\Phi_{\mathcal B}^\ast \, \bm\Phi_{\mathcal B} \Big\|_{2}
        = \Big\|\Big[ \frac 1n \sum_{k=0}^{n-1} \frac{\exp(2\pi\mathrm i\tfrac kn\langle\bm h'-\bm h,\bm z\rangle)}{\sqrt{r(\bm h) \, r(\bm h')}} \Big]_{\bm h,\bm h'\notin\mathcal B} \Big\|_{2} \,.
    \end{equation*}
    By the character property $\frac 1n \sum_{k=0}^{n-1} \exp(2\pi\mathrm i m\tfrac kn) = \delta_{m\equiv_n 0}$ this evaluates to
    \begin{equation*}
        \frac 1n \|\bm\Phi_{\mathcal B}\|_{2}^2
        = \Big\|\Big[ \frac{\delta_{\langle\bm h,\bm z\rangle \equiv_n \langle\bm h',\bm z\rangle}}{\sqrt{r(\bm h) \, r(\bm h')}} \Big]_{\bm h,\bm h'\notin\mathcal B} \Big\|_{2} \,.
    \end{equation*}
    Now we split this matrix into its diagonal and off-diagonal terms
    \begin{equation*}
        \frac 1n \|\bm\Phi_{\mathcal B}\|_{2}^2
        \le \Big\| \Diag\Big( \frac{1}{r(\bm h)}\Big)_{\bm h\notin\mathcal B} \Big\|_{2}
        + \Big\|\Big[ \frac{\delta_{\bm h\neq\bm h'} \delta_{\langle\bm h,\bm z\rangle \equiv_n \langle\bm h',\bm z\rangle}}{\sqrt{r(\bm h) \, r(\bm h')}} \Big]_{\bm h,\bm h'\notin\mathcal B} \Big\|_{2} \,.
    \end{equation*}
    The spectral norm of the diagonal matrix is known and for the off-diagonal matrix we estimate it by the Frobenius norm
    \begin{align*}
        \frac 1n \|\bm\Phi_{\mathcal B}\|_{2}^2
        &\le \sup_{\bm h\notin\mathcal B} \frac{1}{r(\bm h)}
        + \Big\|\Big[ \frac{\delta_{\bm h\neq\bm h'} \delta_{\langle\bm h,\bm z\rangle \equiv_n \langle\bm h',\bm z\rangle}}{\sqrt{r(\bm h) \, r(\bm h')}} \Big]_{\bm h,\bm h'\in\mathds Z^d} \Big\|_{F} \\
        &= \sup_{\bm h\notin\mathcal B} \frac{1}{r(\bm h)}
        + \Big( \sum_{\bm h\in\mathds Z^d} \sum_{\substack{\bm h'\in\mathds Z^d \\ \bm h'\neq\bm h}} \frac{\delta_{\langle\bm h,\bm z\rangle \equiv_n \langle\bm h',\bm z\rangle}}{r(\bm h) \, r(\bm h')} \Big)^{1/2} \\
        &= \sup_{\bm h\notin\mathcal B} \frac{1}{r(\bm h)}
        + \Big( \sum_{\bm h\in\mathds Z^d} \frac{1}{r(\bm h)} \sum_{\substack{\bm\ell\in\mathds Z^d\setminus\{\bm 0\} \\\langle\bm\ell,\bm z\rangle\equiv_n 0}} \frac{1}{r(\bm h+\bm\ell)} \Big)^{1/2} \,,
    \end{align*}
    where the second term is exactly $\sqrt{\mathcal S_n(\bm z)}$ as claimed.
\end{proof} 

Now we fix the frequency set $\mathcal B$ to be as in \Cref{classicalrecon}, where we have guaranteed reconstructing property, allowing us to omit it in the assumptions.

\begin{theorem}\label{boundwithS} Let $\bm X = \{\frac kn\bm z\modone\}_{k=0}^{n-1} \subseteq\mathds T^d$ be a lattice and $\mathcal S_n(\bm z)$ as in \eqref{eq:S}.
    Further, let $H$ be the function space defined in \eqref{eq:H} and fix
    \begin{equation*}
        \mathcal B \coloneqq \Big\{\bm h\in\mathds Z^d : r(\bm h) \,\le\,
        \frac{1}{2\sqrt{\mathcal S_n(\bm z)}} \Big\} \,.
    \end{equation*}
    For $t\ge 4$, let $J\subseteq\{0, \dots, n-1\}$ be a multiset of uniformly i.i.d.\ drawn integers with
    \begin{equation*}
|J| := \lceil 12 \, |\mathcal B| \, (\log|\mathcal B| + t)\rceil \,.
    \end{equation*}
    Then the least squares approximation defined in \eqref{eq:lsqr} using samples on the subsampled lattice $\bm X_J$ satisfies with probability exceeding $1-3\exp(-t)$ that
    \begin{equation*}
        \ewc(S_{\mathcal B}^{\bm X_J})^2
        \le 14 \sqrt{\mathcal S_n(\bm z)}
        + \frac{7}{|\mathcal B|} \sum_{\bm h\notin\mathcal B} \frac{1}{r(\bm h)} \,.
    \end{equation*}
\end{theorem} 

\begin{proof} By \Cref{classicalrecon} we know that $\bm X$ has the reconstructing property for $\mathcal B$.
    Thus, we can apply \Cref{withPhi}, where we plug in \Cref{nokings} for estimating $\|\bm\Phi_{\mathcal B}\|_{2}^2$ and obtain with the desired probability
    \begin{equation*}
        \ewc(S_{\mathcal B}^{\bm X_J})^2
        \le \sup_{\bm h\notin\mathcal B} \frac{1}{r(\bm h)}
        + \frac{7}{|\mathcal B|} \sum_{\bm h\notin\mathcal B} \frac{1}{r(\bm h)}
        + 4 \Big(
        \sup_{\bm h\notin\mathcal B} \frac{1}{r(\bm h)}
        + \sqrt{\mathcal S_n(\bm z)} \Big) \,.
    \end{equation*}
    By the definition of $\mathcal B$, we have $\sup_{\bm h\notin\mathcal B} (1/r(\bm h)) \le 2\sqrt{\mathcal S_n(\bm z)}$, which, when plugged in the above estimate, yields the assertion.
\end{proof} 

\begin{remark}\label{workstoo} Following the same proof, it is easily seen, that \Cref{boundwithS} still holds if we replace $\mathcal S_n(\bm z)$ by any upper estimate, even though it changes the definition of $\mathcal B$ and with that the underlying method.
\end{remark} 

Note that the occurrence of the quantity $(1/|\mathcal B|) \sum_{\bm h\notin\mathcal B} (1/r(\bm h))$ in \Cref{boundwithS} is natural, since there exist function spaces where this is a lower bound on the sampling numbers with $|\mathcal B|$ samples, cf.~\cite[Theorem~2]{HKNV22}.
Comparing the full lattice algorithm \Cref{classicalbound} with \Cref{boundwithS}, both bounds depend on the term $\sqrt{\mathcal S_n(\bm z)}$.
The difference is in the number of used samples being $|J|$ in the new algorithm instead of $n$, with $|J|$ usually far less than $n$, as we will see in \Cref{sec:korobov}.
 \section{Application to weighted Korobov spaces}\label{sec:korobov} 

In this section we apply the general theory to weighted Korobov spaces to obtain double the main rate of convergence for the subsampled lattice in comparison to the full lattice.
For dimension $d\in\mathds N$, smoothness $\alpha > 1/2$, and weight parameters $\bm\gamma = \{\gamma_{\mathfrak u}\}_{\mathfrak u\subseteq\mathds N}$, the \emph{weighted Korobov space} is defined by
\begin{equation}\label{eq:korobov}
    H_{d,\alpha,\bm\gamma}
    \coloneqq \Big\{
        f\colon\mathds T^d\to\mathds C :
        \|f\|_{d,\alpha,\bm\gamma}^2
        = \sum_{\bm h\in\mathds Z^d} r_{d,\alpha,\bm\gamma}(\bm h) \, |\hat f_{\bm h}|^2
        < \infty
    \Big\} \,,
\end{equation}
with
\begin{equation}\label{eq:rgamma}
    r_{d,\alpha,\bm\gamma}(\bm h)
    \coloneqq \gamma_{\Supp(\bm h)}^{-1} \prod_{j\in\Supp(\bm h)} |h_j|^{2\alpha}
    \quad\text{and}\quad
    \Supp(\bm h) \coloneqq \Big\{ j \in\{1,\dots,d\} : h_j\neq 0 \Big\} \,.
\end{equation}
Here, the weight $\gamma_{\mathfrak u}$ controls the importance of the part of the function $f$ which only depends on the variables $\bm x_{\mathfrak u} = [x_j]_{j\in\mathfrak u}$.
We fix $\gamma_{\emptyset} \coloneqq 1$ such that $\|1\|_{d,\alpha,\bm\gamma} = \|1\|_{L_2}$.
For integer $\alpha\in\mathds N$ it is known that
\begin{equation*}
    \|f\|_{d,\alpha,\bm\gamma}^2
    =\!\! \sum_{\mathfrak u\subseteq \{1, \dots, d\}}
    \!\frac{1}{(2\pi)^{\alpha|\mathfrak u|}}
    \frac{1}{\gamma_{\mathfrak u}}
    \int_{[0,1]^{|\mathfrak u|}}
    \!\Big|
    \int_{[0,1]^{d-|\mathfrak u|}}
    \!\Big(\prod_{j\in\mathfrak u} \frac{\partial}{\partial x_j}\Big)^\alpha
    f(\bm x)
    \;\mathrm d\bm x_{\{1, \dots, d\}\setminus\mathfrak u}
    \Big|^2
    \mathrm d\bm x_{\mathfrak u} .
\end{equation*}
Korobov spaces describe a wide class of functions, with various forms of weights $\bm\gamma$ used throughout literature, such as
\begin{itemize}
\item
    $\gamma_{\mathfrak u} = 1$ for all $\mathfrak u\subseteq\mathds N$: these spaces coincide with classical (unweighted) periodic Sobolev spaces with dominating mixed smoothness, see e.g., \cite{DTU18};
\item
    $\gamma_{\mathfrak u} = \prod_{j\in\mathfrak u}\gamma_j$ with a positive sequence $(\gamma_j)_{j\ge 1}$: these are called product weights, see e.g.,~\cite{Sloan1998, Woniakowski2009};
\item
    $\gamma_{\mathfrak u} = \Gamma_{|\mathfrak u|}$ for a sequence $(\Gamma_\ell)_{\ell\ge 0}$: they are called order dependent weights, see e.g.,~\cite{Dick2006};
\item
    combining product and order dependent weights yield POD weights $\gamma_{\mathfrak u} = \Gamma_{|\mathfrak u|}\prod_{j\in\mathfrak u}\gamma_j$, see e.g., \cite{Kuo2012, Graham2014, Kuo2016, Graham2018};
\item
    including the smoothness parameter yields smoothness driven product and order dependent (SPOD) weights  $\gamma_{\mathfrak u} = \sum_{\bm\nu_{\mathfrak u}\in\{1, \dots, \alpha\}^{|\mathfrak u|}} \Gamma_{\|\bm\nu_{\mathfrak u}\|_1}\prod_{j\in\mathfrak u}\gamma_{j,\nu_j}$, see e.g., \cite{DKLNS14,Dick2016,Kaarnioja2020}.
\end{itemize}

For the classical periodic Sobolev spaces of dominating mixed smoothness, it is known that the worst-case error behaves asymptotically as $n^{-\alpha}(\log n)^{\alpha(d-1)}$ when the approximation is based on $n$ measurements from arbitrary linear functionals.
This includes function evaluations as well as Fourier coefficients, see e.g., \cite{M62} or \cite[Section~4.1]{DTU18}.
It was finally shown in \cite{DKU23} that the same rate of convergence can be attained when the approximation was limited to use function evaluations.
This result is a non-constructive existence result for the points used in the approximation.

On the other hand, it was shown in \cite[Proposition~3.3]{BKUV17} that for an approximation based on a full lattice the worst-case error can be bounded from below by $n^{-\alpha/2}$.
Thus, using the full lattice halves the polynomial rate of convergence.

It was already shown in~\cite{BKPU22}, that the optimal polynomial rate can be recovered using a subset of the initial lattice given that the initial lattice is large enough.
In particular, in~\cite{BKPU22} a lattice of size $n \gtrsim |J|^{2/(1-1/(2\alpha))}$ is necessary.
We will show in~\Cref{koroboverror} and~\Cref{cor} that the optimal polynomial rate $|J|^{-\alpha+\varepsilon}$ can be recovered using only an initial lattice of size $n \lesssim\, |J|^{2/\sqrt{1-\varepsilon/\alpha}}$ for any $0<\varepsilon<\alpha$.

In order to apply our general result on subsampled lattices, we first present known results for the full lattice.
Note that the parameter $\alpha$ in \cite{KMN24} is replaced by $2\alpha$ in this paper.

\begin{theorem}[{\cite[Theorem~3.3]{KMN24}}]\label{Skorobov} Given $n\ge 2$, $d\in\mathds N$, $\alpha>1/2$ and weights $\bm\gamma = \{\gamma_{\mathfrak u}\}_{\mathfrak u\subseteq\mathds N}$, the generating vector $\bm z$ obtained from the CBC construction following \cite[Algorithm~3.2]{KMN24} satisfies \begin{equation*}
      \frac{2\zeta(2\alpha)\gamma_{\{1\}}}{n^{2\alpha}} 
      \le \mathcal S_n(\bm z)
        \le \Big( \frac{\tau_\lambda \, C_\lambda^2}{\varphi(n)}
        \Big)^{1/\lambda} 
        \quad\mbox{for all}\quad \tfrac{1}{2\alpha} < \lambda \le 1\,,
    \end{equation*}
    where $\varphi(n) = |\{z\in \{1,\ldots, n\} : \gcd(z,n)=1\}|$ is the Euler totient function (e.g., $\varphi(n) = n-1$ for $n$ prime), $\zeta(x) = \sum_{h=1}^{\infty}h^{-x}$ is the Riemann zeta function, $\tau_\lambda = 2^{4\alpha\lambda+1}+1$, and
    \begin{equation}\label{eq:C}
        C_\lambda
        = C_{d,\alpha,\bm\gamma,\lambda}
        \coloneqq \sum_{\mathfrak u\subseteq\{1,\dots,d\}} \max\{1,|\mathfrak u|\} \, \gamma_{\mathfrak u}^{\lambda} \, (2\zeta(2\alpha\lambda))^{|\mathfrak u|} \,.
    \end{equation}
\end{theorem} 

\begin{proof} The trivial lower bound is obtained by keeping only the $\bm h = \bm 0$ and $\bsell = [\ell_1,0,\ldots,0]$ terms in \eqref{eq:S} together with $z_1 = 1$.
A CBC upper bound for $\mathcal S_n(\bm z)$ was first proved in \cite{DKKS07} for product weights, in \cite{CKNS20} for general weights, both for prime~$n$, and in \cite{KMN24} for composite $n$.
\end{proof} 

Applying the estimate on $\mathcal S_n(\bm z)$ from \Cref{Skorobov} to \Cref{classicalbound} achieves the order of convergence $n^{-\alpha/2+\varepsilon}$ for the full lattice, which is half the optimal order.

\begin{lemma}[{\cite[Lemmata~5.2 and 5.3]{KMN24}}]\label{muzofm} Let $M\ge 1$, $\alpha>1/2$, and
    \begin{equation*}
        \mathcal A_{d,\alpha,\bm\gamma,M}
        \coloneqq \Big\{\bm h\in\mathds Z^d : r_{d,\alpha,\bm\gamma}(\bm h) \le M \Big\} \,.
    \end{equation*}
    Then we have for $1/(2\alpha) < \lambda < 1$
    and $C_\lambda=C_{d,\alpha,\bm\gamma,\lambda}$ from \eqref{eq:C} that
    \begin{equation*}
        (\gamma_{\{1\}} M)^{1/(2\alpha)}
        \le |\mathcal A_{d,\alpha,\bm\gamma,M}|
        \le C_\lambda M^{\lambda}
    \end{equation*}
    and
\begin{equation*}
        \frac{1}{|\mathcal A_{d,\alpha,\bm\gamma,M}|} \sum_{\bm h\notin \mathcal A_{d,\alpha,\bm\gamma,M}} \frac{1}{r_{d,\alpha,\bm\gamma}(\bm h)}
        \le
        \Big(
        \frac{ C_\lambda^{1/\lambda} }{ \gamma_{\{1\}}^{1/(2\alpha\lambda)} }
        \cdot
        \frac{\lambda }{1-\lambda}
        \Big)
        \Big(\frac 1M\Big)^{-\frac{1}{2\alpha\lambda}} \,.
    \end{equation*}
\end{lemma} 

In the original \cite[Lemma~5.2]{KMN24} the constant $C_\lambda$ equals $\sum_{\mathfrak u\subseteq\{1,\dots,d\}} \gamma_{\mathfrak u}^{\lambda} \, (2\zeta(2\alpha\lambda))^{|\mathfrak u|}$.
Here we used an upper estimate \eqref{eq:C} to simplify our presentation.

We now present the main result of this section which bounds the worst-case error for a subsampled lattice achieving optimal polynomial rate using the smallest initial lattice size.

\begin{theorem}\label{koroboverror} Let $H$ be the Korobov function space defined in \eqref{eq:korobov} with $\alpha>1/2$ and weights $\bsgamma = (\bsgamma_\setu)_{j\in\setu}$.
    Let $\bm X = \{\frac kn\bm z\modone\}_{k=0}^{n-1} \subseteq\mathds T^d$ be a lattice with $n\ge 7$
    sufficiently large so that $\mathcal S_n(\bsz) = \mathcal S_{d,\alpha,\bsgamma,n}(\bsz)\le 1/4$, and let
    \begin{equation*}
        \mathcal B \coloneqq \mathcal B_{d,\alpha,\bsgamma,n}
        \coloneqq \Big\{ \bm h\in\mathds Z^d : r_{d,\alpha,\bm\gamma}(\bm h) \le
        \frac{1}{2\sqrt{\mathcal S_n(\bm z)}}
        \Big\} \,.
    \end{equation*}
    For $t\ge 4$, let $J\subseteq\{0, \dots, n-1\}$ be a multiset of uniformly i.i.d.\ drawn integers with
    \begin{equation*}
        |J|
        \coloneqq \lceil 12\,|\mathcal B|\,(\log|\mathcal B| + t) \rceil \,.
    \end{equation*}
    Then the least squares approximation defined in \eqref{eq:lsqr} using samples on the subsampled lattice $\bm X_J$ satisfies with probability exceeding $1-3\exp(-t)$ that
    \begin{equation} \label{eq:finalE}
        \ewc(S_{\mathcal B}^{\bm X_J})^2
        \le 7\Big(1+
        \frac{ C_\lambda^{1/\lambda} }{ \gamma_{\{1\}}^{1/(2\alpha\lambda)} }
        \cdot
        \frac{\lambda }{1-\lambda}
        \Big)\,(13\,C_\lambda)^{\frac{1}{2\alpha\lambda^2}}\,
        \Big(\frac{\log|J|+t}{|J|}\Big)^{\frac{1}{2\alpha\lambda^2}} 
    \end{equation}
    for all $1/(2\alpha) < \lambda < 1$, with $C_\lambda=C_{d,\alpha,\bm\gamma,\lambda}$ from \eqref{eq:C}.
    Moreover, if $\bm z$ is constructed according to \cite[Algorithm~3.2]{KMN24}, we have (with $\varphi(n)$ again denoting the Euler totient function)
    \begin{equation} \label{eq:finalJ}
        (\varphi(n))^{\frac{1}{4\alpha\lambda}}\,\log\varphi(n)
        \lesssim |J|
        \lesssim n^{\alpha\lambda} \log n \,,
    \end{equation}
    and hence
    \begin{equation} \label{eq:finalN}
        \Big(\frac{|J|}{\log|J|}\Big)^{\frac{1}{\alpha\lambda}}
        \lesssim n \quad\mbox{and}\quad
        \varphi(n)
        \lesssim \Big(\frac{|J|}{\log|J|}\Big)^{4\alpha\lambda} \,.
    \end{equation}
\end{theorem} 

\begin{proof} For the index set $\calB$ with radius $M := 1/(2\sqrt{\mathcal S_n(\bsz)})\ge 1$, \Cref{muzofm} gives for all $1/(2\alpha) < \lambda < 1$,
    \begin{equation*}
        (\gamma_{\{1\}}\,M)^{1/(2\alpha)} \le |\calB| \le C_\lambda\,M^\lambda
      \quad\Leftrightarrow\quad
      \frac{\gamma_{\{1\}}}{|\calB|^{2\alpha} }
      \le \frac{1}{M} 
      \le \Big(\frac{C_\lambda}{|\calB|}\Big)^{1/\lambda}\,.
    \end{equation*}
    \Cref{boundwithS} and \Cref{muzofm} then yield
    \begin{align*}
        \ewc(S_{\mathcal B}^{\bm X_J})^2
        &\le \frac{7}{M} 
        +7\Big(
        \frac{ C_\lambda^{1/\lambda} }{ \gamma_{\{1\}}^{1/(2\alpha\lambda)} }
        \cdot
        \frac{\lambda }{1-\lambda}
        \Big)
        \Big(\frac{1}{M}\Big)^{\frac{1}{2\alpha\lambda}} \\
        &\le 7\Big(1+
        \frac{ C_\lambda^{1/\lambda} }{ \gamma_{\{1\}}^{1/(2\alpha\lambda)} }
        \cdot
        \frac{\lambda }{1-\lambda}
        \Big)
        \Big(\frac{C_\lambda}{|\calB|}\Big)^{\frac{1}{2\alpha\lambda^2}}\,,
    \end{align*}
    from which the bound \eqref{eq:finalE} is obtained using the upper bound on $1/|\calB|$ in \eqref{eq:JB}.

    Next we relate $|\calB|$ to the bounds on $M = 1/(2\sqrt{\mathcal S_n(\bm z)})$ from \Cref{Skorobov}, to obtain
    \begin{align*}
        \Big(\frac{\varphi(n)}{\tau_\lambda\, C_\lambda^2}\Big)^{\frac{1}{4\alpha\lambda}}\Big(\frac{\gamma_{\{1\}}}{2}\Big)^{\frac{1}{2\alpha}}
        \le \Big(\frac{\gamma_{\{1\}}}{2\sqrt{\mathcal S_n(\bm z)}}\Big)^{\frac{1}{2\alpha}}
        \le |\mathcal B|
        &\le C_\lambda\Big(\frac{1}{2\sqrt{\mathcal S_n(\bm z)}}\Big)^\lambda \\
        &\le \frac{C_\lambda\,n^{\alpha\lambda}}{(8\zeta(2\alpha)\gamma_{\{1\}})^{\lambda/2}} \,.
    \end{align*}
    Since $12\,|\mathcal B|\,(\log|\mathcal B|+t) \le |J| \le 13\,|\mathcal B|\,(\log|\mathcal B| + t)$, we obtain \eqref{eq:finalJ} and \eqref{eq:finalN}.
\end{proof} 

Note that for small $n$ it is possible for the theoretically defined value of $|J|$ to be bigger than $n$. Additionally, the constants in the asymptotic bounds \eqref{eq:finalJ} and \eqref{eq:finalN} are quite loose. Later in our numerical experiments we will choose instead $|J| = \lceil\sqrt{n}\,\log(n)\rceil$.

The following Corollary shows that the asymptotic rates in \Cref{koroboverror} are optimal.

\begin{corollary}\label{cor} Under the setting of \Cref{koroboverror} with prime $n$, we have the asymptotic rates
    \begin{equation}
    \label{eq:err_no_log}
        \ewc(S_{\mathcal B}^{\bm X_J})
        \lesssim |J|^{-\alpha+\varepsilon}
        \quad\text{with}\quad
        |J|^{2\sqrt{1-\varepsilon/\alpha}}
        \lesssim
        n
        \lesssim
        |J|^{2/\sqrt{1-\varepsilon/\alpha}} \end{equation}
    for all $0<\varepsilon < \alpha$.
    This is optimal in that any method utilizing $|J|$ information cannot have a better rate than $|J|^{-\alpha}$, and the rate given implies an initial lattice size of $n\gtrsim |J|^{2-2\varepsilon/\alpha}$.
\end{corollary} 

\begin{proof} Rall that \eqref{eq:finalE} and \eqref{eq:finalN} in \Cref{koroboverror} hold for all parameter $\lambda$ satisfying $1/(2\alpha) < \lambda < 1$. We reparametrize in terms of arbitrary $0 <\varepsilon < \alpha$ by choosing $\lambda$ such that 
\[
  \alpha - \frac{1}{4\alpha\lambda^2} < \varepsilon,
\]
which is possible since $\lambda>1/(2\alpha)$ implies that the left-hand side can be made arbitrarily small. Thus from \eqref{eq:finalE} we have{}
\begin{equation*}
        \ewc(S_{\mathcal B}^{\bm X_J})
        \lesssim \Big(\frac{\log|J|}{|J|}\Big)^{\frac{1}{4\alpha\lambda^2}}
        = \Big(\frac{\log|J|}{|J|}\Big)^{\alpha - (\alpha-\frac{1}{4\alpha\lambda^2})}
        \lesssim|J|^{-\alpha+\varepsilon} \,.
    \end{equation*}
Moreover, from \eqref{eq:finalN} we have
    \begin{equation*}
        n \gtrsim \Big(\frac{|J|}{\log|J|}\Big)^{\frac{1}{\alpha\lambda}}
= \Big(\frac{|J|}{\log|J|}\Big)^{2\sqrt{1 - (\alpha-\frac{1}{4\alpha\lambda^2})/\alpha}}
         \gtrsim |J|^{2\sqrt{1-\varepsilon/\alpha}}
    \end{equation*}
    and, using $\varphi(n) = n -1$ for $n$ prime,
    \begin{equation*}
        n \lesssim \varphi(n) \lesssim \Big(\frac{|J|}{\log|J|}\Big)^{4\alpha\lambda}
= \Big(\frac{|J|}{\log|J|}\Big)^{\frac{2}{\sqrt{1-(\alpha-\frac{1}{4\alpha\lambda^2})/\alpha}}}
        \lesssim |J|^{\frac{2}{\sqrt{1-\varepsilon/\alpha}}}.
    \end{equation*}
    This proves the assertion \eqref{eq:err_no_log}.

    The optimality of the worst-case error follows from bounds on the linear width in \cite[Theorem~4.45]{DTU18}.
    The optimality for the initial lattice size follows from \cite{BKUV17}, since an algorithm using a subsampled lattice is inherently still using the lattice of size~$n$, and achieving the rate~$|J|^{-\alpha+\varepsilon}$ implies $n^{-\alpha/2} \lesssim |J|^{-\alpha+\varepsilon}$, yielding the lower bound.
\end{proof} 

We remark that the $\varepsilon$ in \Cref{cor} incorporates both the gap that naturally arises when approximating in Korobov spaces and the additional logarithmic gap stemming from the random nature of our method.
In contrast, \Cref{koroboverror} clearly distinguishes between these two effects.
 \section{Implementation and computational cost}\label{sec:implementation} 

The goal of subsampling from a lattice is to make use of the underlying structure, where fast and memory-efficient algorithms are applicable.
Without that structure uniformly generated random points would suffice for a similar error bound, cf.\ \cite{KU20, KUV21}.

Both the least squares approximation and the kernel method involve solving a system of equations.
For the full lattice closed form solutions are known, but for the subsampled lattice it seems that an analytical inverse is not available.
Direct solvers would have a cubic runtime with respect to the number of points $|J|$.
Instead, we will apply iterative methods: the least squares (\texttt{LSQR}) \cite{PS82} and the conjugate gradient (\texttt{CG}) \cite{Gre97} methods, where only matrix-vector products are needed.
The following result bounds the number of necessary iterations depending on the desired accuracy.

\begin{theorem}[{\cite[Theorem~3.1.1]{Gre97}}]\label{iterations} Let $\bm A$ be a Hermitian positive definite matrix with condition number $\kappa_2(\bm A) = \lambda_{\max}(\bm A)/\lambda_{\min}(\bm A)$.
    Further, let $\bm x^{(i)}$ be the $i$-th iterate of the conjugate gradient (\texttt{CG}) method applied to the system of equations $\bm A\bm{\tilde x} = \bm b$.
    Then it holds
    \begin{equation*}
        \frac{\|\bm x^{(i)} - \bm{\tilde x}\|_2}{\sqrt{\kappa_2(\bm A)}\|\bm{\tilde x}\|_2}
        \le \! \sqrt{\frac{(\bm x^{(i)}-\bm{\tilde x})^\ast \bm A \, (\bm x^{(i)}-\bm{\tilde x})}{(\bm{\tilde x})^\ast \bm A \, \bm{\tilde x}}}
        \!\le 2 \Big(\frac{\sqrt{\kappa_2(\bm A)} - 1}{\sqrt{\kappa_2(\bm A)} + 1}\Big)^i
        \!\le 2 \exp\Big(\!\frac{-2i}{\sqrt{\kappa_2(\bm A)}}\Big) .
    \end{equation*}
\end{theorem} 

\begin{proof} The first inequality follows from $\lambda_{\min}(\bm B)\|\bm x\|_2 \le \|\bm B\bm x\|_2 \le \lambda_{\max}(\bm B)\|\bm x\|_2$ for any Hermitian positive definite matrix $\bm B$ and vector $\bm x$.
    The second inequality is the statement of \cite[Theorem~3.1.1]{Gre97}.
    For the last inequality, we need that for $x\ge 1$
    \begin{equation*}
        \frac{x-1}{x+1} \le \exp\Big(\!-\frac 2x\,\Big) \,,
    \end{equation*}
    or equivalently
    \begin{equation*}
        g(x) \coloneqq \log \Big(\frac{x-1}{x+1}\Big) + \frac 2x \le 0 \,.
    \end{equation*}
    For $x$ approaching $1$ the statement is true.
    For $x>1$, we have $g'(x) > 0$.
    Together with $\lim_{x\to\infty}g(x) = 0$ this gives the assertion.
\end{proof} 

The \texttt{LSQR} method is analytically equivalent to applying \texttt{CG} to the normal equation but in a numerically more stable manner, cf.\ \cite{PS82}.
Thus the iteration bound in \Cref{iterations} holds for \texttt{LSQR} with the minimal and maximal eigenvalue replaced by the squared singular numbers.

\subsection{Fast least squares approximation with subsampled lattices}\label{sec:ilsqr}

In this section we discuss how to compute the least squares approximation for a subsampled lattice making use of the Fast Fourier Transform (\texttt{FFT}).
Let $\mathcal B \coloneqq \{\bm h\in\mathds Z^d : r(\bm h)\le M\}$ and
\begin{equation*}
    S_{\mathcal B}^{\bm X_J}f
    = \sum_{\bm h\in\mathcal B}\hat g_{\bm h} \exp(2\pi\mathrm i\langle\bm h,\cdot\rangle)
\end{equation*}
be defined in \eqref{eq:lsqr}.
The approximation is computed once the Fourier coefficients $\bm{\hat g}_{\mathcal B} = (\hat g_{\bm h})_{\bm h\in B}$ representing the approximation are obtained.
Essential is the set of frequencies $\mathcal B$.
Because of the downward-closed structure, this can be set up incrementally in dimension.
For the step from dimension $j$ to $j+1$, we have frequencies supported in the first $j$ dimensions and for every one of them we seek all $h_{j+1}$ such that $r_{d,\alpha,\bm\gamma}([h_1, \dots, h_j, h_{j+1},$ $ 0, \dots, 0]) \le M$.
This has a computational cost of $\mathcal O(d\,|\mathcal B|)$.
We then compute the coefficients by solving the system of equations in \Cref{lsqrcoeffs} via \texttt{LSQR}.

By \eqref{eq:sepultura} in \Cref{constructions}, we have, with $|J|$ sufficiently large and high probability, that $\kappa_2(\bm L_{J,\mathcal B}^{\ast} \, \bm L_{J,\mathcal B}) = \lambda_{\max}(\bm L_{J,\mathcal B}^{\ast} \, \bm L_{J,\mathcal B})/\lambda_{\min}(\bm L_{J,\mathcal B}^{\ast} \, \bm L_{J,\mathcal B}) \le 3$.
Fixing the desired accuracy to machine precision $\texttt{eps} = 10^{-16}$, we obtain by \Cref{iterations} a fixed number of iterations for \texttt{LSQR}.
Each iteration uses one matrix-vector product with $\bm L_{J,\mathcal B}$ and one with $\bm L_{J,\mathcal B}^{\ast}$.
For the full lattice $J = \{1,\dots,n\}$ this matrix-vector product simplifies to a one-dimensional \texttt{FFT} with a computational cost of
\begin{equation*}
    \mathcal O(n\log n + |\mathcal B|) \text{ flops and } \mathcal O(n) \text{ memory requirements} \,,
\end{equation*}
which is known as the Lattice Fast Fourier Transform (\texttt{LFFT}), cf.\ \cite{KKP12}.
The multiplication with the matrix corresponding to the subsampled lattice can be done using the full lattice by
\begin{equation}\label{eq:LP}
    \bm L_{J,\mathcal B}
    = \bm P\bm L_{\mathcal B}
    \quad\text{with}\quad
    \bm P 
    = \left(\begin{matrix}\\\\\\\\\end{matrix}\right. \begin{matrix} \\ 1 & 0 & 0 & \dots & 0 & 0 & 0 \\ 0 & 0 & 1 & \dots & 0 & 0 & 0 \\ & \vdots &&&& \vdots & \\ 0 & 0 & 0 & \dots & 0 & 1 & 0 \\ \textcolor{gray}{j_1} && \textcolor{gray}{j_2} &&& \textcolor{gray}{j_{|J|}} \end{matrix} \left.\begin{matrix}\\\\\\\\\end{matrix}\right)
    \textcolor{gray}{\begin{matrix}1\\2\\\vdots\\|J|\end{matrix}}
        \in\{0,1\}^{|J|\times n} \,,
\end{equation}
which has thus the same computational cost.
Note, we essentially have an \texttt{FFT} where we only use parts of the result.
This is known as ``pruned \texttt{FFT}'', which are partial \texttt{FFT}s with cost $\mathcal O(n\log|J|)$.
As current \texttt{FFT} implementations are highly tuned this gain is not observed in practice.

With $|J| \sim \sqrt n \log n$ and $|\mathcal B| \sim \sqrt{n}$ as in \Cref{koroboverror} we obtain the computational cost in the upper two rows of \Cref{tab:cost} when using the \texttt{LFFT} compared to the naive matrix-vector product.
The iteration count does not appear here as it is bounded by a general constant independent of the number of points.

\begin{table} \centering
    \scalebox{0.93}{
    \bgroup
    \def\arraystretch{1.3}
    \begin{tabular}{|p{150pt}|l|l|}
        \hline
        \textbf{Subsampled method} & \textbf{Flops} & \textbf{Memory} \\
        \hline
        least squares approx.\ (naive) & $\mathcal O(|J||\mathcal B|) = \mathcal O(n\log n)$ & $\mathcal O(|J||\mathcal B|) = \mathcal O(n\log n)$ \\
        \hline
        least squares approx.\ using \texttt{LFFT} & $\mathcal O(n\log n)$ & $\mathcal O(n)$ \\
        \hline
        kernel method (naive) & $\mathcal O(r \, |J|^2) = \mathcal O(r \, n\log n)$ & $\mathcal O(|J|^2) = \mathcal O(n\log n)$ \\
        \hline
        kernel method using \texttt{FFT} & $\mathcal O(r \, n\log n)$ & $\mathcal O(n)$ \\
        \hline
    \end{tabular}
    \egroup
    }
    \vspace{10pt}
    \caption{Computational cost for the least squares approximation $S_{\mathcal B}^{\bm X_J}$ and the kernel method $A_{\text{ker}}^{\bm X_J}$ with a subsampled lattice of size $|J| = \lceil\sqrt{n}\,\log n\rceil$ utilizing the FFT compared to the naive matrix-vector product.
    Here $r$ is the number of iterations depending on the condition number of the kernel matrix.}\label{tab:cost}
\end{table} 

Note, when using the full lattice an iterative solver is not needed, since we have by the reconstructing property $(\bm L_{\mathcal B}^{\ast} \, \bm L_{\mathcal B}^{\ast})^{-1} = (1/n)\,\bm I$.
This can be also viewed as \texttt{LSQR} converging with one iteration because $\kappa_2(\bm L_{\mathcal B}^{\ast} \, \bm L_{\mathcal B}) = 1$ in \Cref{iterations}.
Thus, using the full lattice has computational cost $\mathcal O(n\log n)$ as well.

\subsection{Fast kernel method with subsampled lattices}\label{sec:ikernel} 

We further comment on the kernel method, which we introduced in \Cref{sec:classical} for the full lattice.
Since $H_{d, \alpha, \bm\gamma}$ is a reproducing kernel Hilbert space, it has an associated reproducing kernel $K(\cdot,\cdot) = K_{d,\alpha,\bm\gamma}(\cdot,\cdot)$.
For integer smoothness $\alpha\in\mathds N$, a simple closed form is known:
\begin{equation*}
    K(\bm x,\bm y)
    = \!\! \sum_{\mathfrak u\subseteq\{1,\dots,d\}} \!\! \gamma_{\mathfrak u} \prod_{j\in\mathfrak u} \eta_{\alpha}(x_j,x_j')
    \;\;\text{with}\;\;
    \eta_{\alpha}(x,x')
    = \frac{(2\pi)^{2\alpha} B_{2\alpha}((x-x')\modone)}{(-1)^{\alpha+1}(2\alpha)!} \,,
\end{equation*}
where $B_{2\alpha}$ is the Bernouilli polynomial of degree $2\alpha$.

The \emph{kernel approximation} of $f$ based on the sampling set $\bm X_J$ is defined analogously to \eqref{eq:kernelmethod} by
\begin{equation}\label{eq:subkernelmethod}
    A_{\text{ker}}^{\bm X_J}f
    \coloneqq \sum_{k\in J} a_k \, K(\cdot,\tfrac kn\bm z\modone)
    \quad\text{with}\quad
    \bm a = (a_k)_{k\in J} = \bm K_{J}^{-1}\bm f_J \,,
\end{equation}
where $\bm K_J = [K(\tfrac kn\bm z\modone, \tfrac{k'}{n}\bm z\modone)]_{k,k'\in J}$ is the so-called kernel matrix and $\bm f_J = [f(\tfrac kn\bm z\modone)]_{k\in J}$.

As discussed in \Cref{sec:classical}, the kernel method is optimal in the worst-case setting for any given set of points.
In particular, the bound from \Cref{koroboverror} applies as well, since $\ewc(A_{\text{ker}}^{\bm X_J}) \le \ewc(S_{\mathcal B}^{\bm X_J})$.
Similar to $\bm L_{J,\mathcal B}$, a fast multiplication with the kernel matrix for the full lattice is known, since $\bm K_{\{0,\dots,n-1\}}$ is a circulant matrix, cf.\ \cite[Table~1]{KKKNS21}, with a computational cost of
\begin{equation*}
    \mathcal O(n\log n) \text{ flops and } \mathcal O(n) \text{ memory requirements}.
\end{equation*}
The algorithm for the full lattice can be used for the subsampled kernel matrix as well.
We have
\begin{equation*}
    \bm K_J = \bm P \bm K_{\{0,\dots,n-1\}} \bm P^\ast \,,
\end{equation*}
with $\bm P$ as in \eqref{eq:LP}.
The cost for evaluating the kernel will depend on the structure of the weights $\gamma_{\mathfrak u}$, ranging from linear in $d$ for product weights to quadratic in $d$ for POD and SPOD weights, see \cite[Table~1]{KKKNS21}.

In contrast to the least squares matrix $\bm L_{J,\mathcal B}$, the condition number of the kernel matrix $\bm K_J$ grows in the number of points, resulting in the need for many iterations or the necessity of preconditioning, cf.\ \cite[Section~12.2]{W04}.
Thus, the number of necessary \texttt{CG} iterations cannot be bounded by a general constant according to \Cref{iterations}.

With $|J| \sim \sqrt n \log n$ as in \Cref{koroboverror}, we obtain the computational cost in the lower two rows of \Cref{tab:cost} when using the \texttt{FFT} compared to the naive matrix-vector product when using $r$ iterations.

\subsection{Comparison to full-lattice methods} 

From \Cref{koroboverror} we know that subsampling allows for a better sampling complexity.
However, picking the best method in regards to computational cost also depends on whether functions evaluations are expensive.
To discuss this balancing, we split the computations into five steps and address their individual computational cost.
This gives the following for the least squares approximation and the kernel method with the full and subsampled lattice, which is summarized in \Cref{tab:cost2}.
\begin{enumerate}
\item
    \emph{Point construction.}
    All methods use the CBC construction from \cite{GS25} with different cost depending on the type of weights as shown in the first row of \Cref{tab:cost2}.
    The additional cost for creating the subsampled indices is negligible.

\item
    \emph{The sampling of the function to approximate.}
    When $f_{\text{cost}}$ models the cost of a single function evaluation, this has to be multiplied by the number of samples used by the approximation.
    For the full lattice these are $n$ points and for the subsampled lattice we have $|J| \sim \sqrt{n}\log n$ points, cf.~\Cref{koroboverror}.

\item
    \emph{The function-independent setup.}
    These computations may depend on a pre-defined function class but can be used for any function thereof.
    This becomes important when multiple instances need to be approximated.

    For the least squares approximation this involves the construction of the frequency index set $\mathcal B$ and the computation of the inner products with the generating vector~$\bm z$.
    As discussed in \Cref{sec:ilsqr} this can be done incrementally in dimension with a cost of $d\,|\mathcal B|$, where $|\mathcal B| \sim \sqrt{n}$ in our case.

    For the kernel method the entries of the kernel matrix $\bm K_{J}$ have to be computed, where the cost of one single kernel evaluation depends on the weight parameters, cf.~\cite[Table~1]{KKKNS21}.
        Because we have $K(\frac kn \bm z\modone,\frac{k'}{n}\bm z\modone) = K(\frac{k-k'}{n}\bm z\modone, \bm 0)$, the number of different values equals $|\{(k-k') \bmod n : k,k'\in J\}|$.
    For the full lattice this evaluates to $n$.
    For the randomly subsampled lattice with $|J|\sim\sqrt n\log n$ we might have $|\{(k-k') \bmod n : k,k'\in J\}| = n$ as well.

\item
    \emph{The computation of the coefficients.}
    At this step the approximation is obtained from the sampled function values using the function-independent setup and encoded in a respective representation.
    The least squares approximation is described by its Fourier coefficients $[ \hat g_{\bm h} ]_{\bm h\in\mathcal B}$ and the kernel method by the coefficients $[a_k]_{k\in J}$ (where $J=\{0,\dots,n-1\}$ for the full lattice).
    These coefficients are defined by the solution of the systems of equations in \Cref{lsqrcoeffs} and \eqref{eq:kernelmethod}, respectively.
    For the full lattice the involved inverse matrices are known in closed form as discussed in \Cref{sec:ilsqr} and \Cref{sec:classical}, which results in a computational cost of an \texttt{FFT} for the computation of the coefficients.
    T least squares matrix $\bm L_{\mathcal B}$ for a full lattice with reconstructing property  is orthogonal and therefore has condition number $\kappa_2(\bm L_{\mathcal B}) = 1$.
    In contrast to that, the kernel matrix $\bm K$ has a bigger condition number, which grows with the smoothness $\alpha$ and number of points $n$, cf.~\cite[Chapter~6]{abithesis}.
    This might require special attention when computing the kernel approximation.{}

    For the subsampled methods we use the iterative schemes described in \Cref{sec:ilsqr} and \Cref{sec:ikernel}.
    In case of the least squares method, the number of iterations is bounded and by a general constant.
    For kernel method more iterations are needed when the number of points increases, as the kernel matrix is worse conditioned for an increasing number of points.
    In \Cref{tab:cost2} the number of iterations for the kernel method is denoted by $r$.
    
\item
    \emph{The evaluation of the approximation.}
    This describes the cost of evaluating the approximation in a single unseen data point $\bm x\in\mathds T^d$.
    For the least squares approximation this involves evaluating the sum
    $(S_{\mathcal B}^{\bm X_J} f)(\bm x) = \sum_{\bm h\in\mathcal B}\hat g_{\bm h} \exp(2\pi\mathrm i\langle\bm h,\bm x\rangle)$
    with a cost of $d\,|\mathcal B|$, with $|\mathcal B|\sim\sqrt n$ in our case.

    For the kernel method we need to evaluate $\sum_{k\in J} a_k \, K(\bm x,\tfrac kn\bm z\modone)$, which needs $|J|$ evaluations of the kernel with $|J| = \{0,\dots,n-1\}$ for the full lattice.
\end{enumerate}
Note that for a given $n$ all methods have the $L_2$ worst-case error proportional to $\sqrt[4]{S_n(\bm z)}$, cf.~\Cref{classicalbound,koroboverror}.
\begin{table} \centering
    \scalebox{0.88}{
    \bgroup
    \def\arraystretch{1.3}
    \begin{tabular}{|l|*{4}{>{\centering\arraybackslash}p{63pt}|}}
        \hline
        \textbf{Computation step} & $S_{\mathcal B}^{\bm X_J}$ & $S_{\mathcal B}^{\bm X}$ & $A_{\text{ker}}^{\bm X_J}$ & $A_{\text{ker}}^{\bm X}$ \\
        \hline
        point construction        & \multicolumn{4}{c|}{$d\,n\log n+d^i \alpha^j n$} \\
        \hline
        sampling                & $f_{\text{cost}}\sqrt{n}\log n$ & $f_{\text{cost}} n$ & $f_{\text{cost}}\sqrt{n}\log n$ & $f_{\text{cost}} n$ \\
        \hline
        function-indep.\ setup  & \multicolumn{2}{c|}{$d\sqrt n$} & \multicolumn{2}{c|}{$d^i\alpha^j n$} \\
        \hline
        computing coefficients  & \multicolumn{2}{c|}{$n\log n$} & $r n\log n$ & $n\log n$ \\
        \hline
        evaluation              & \multicolumn{2}{c|}{$d\sqrt n$} & $d^i\alpha^j\sqrt{n}\log n$ & $d^i\alpha^j n$ \\
        \hline
    \end{tabular}
    \egroup
    }
    \vspace{10pt}
    \caption{Computational cost for the least squares approximation $S_{\mathcal B}^{\bm X_J}$ and the kernel method $A_{\text{ker}}^{\bm X_J}$ with a full and subsampled lattice of size $|J| = \lceil\sqrt{n}\,\log n\rceil$.
    Here $r$ is the number of iterations depending on the condition number of the kernel matrix and $f_{\text{cost}}$ is the cost of a single function evaluation.
    We have $(i,j) = (1,0)$ for product weights, $(i,j) = (2,0)$ for POD weights, and $(i,j) = (2,2)$ for SPOD weights.
    }\label{tab:cost2}
\end{table} 

Apart from user-defined benchmark-functions, the cost of function evaluations $f_{\text{cost}}$ is often expensive, like when solving a PDE or taking in-field measurements, see e.g.~\cite{Kuo2012}.
Then the approximation serves as a surrogate model and a fast evaluation of the approximation is the priority, which favors the methods using the subsampled lattice.
The overload of using an iterative solver is also reasonable, in particular for the least squares approximation, where the number of iterations can be bounded by a general constant.
 \section{Numerical results}\label{sec:numerics} 

In this section we test the proposed algorithms in order to verify our theoretical results by comparing
\begin{itemize}
    \item the classical lattice algorithm \eqref{eq:classical} using the full lattice, with frequency index set $\mathcal A = \{\bm h\in\mathds Z^d : r_{d,\alpha,\bm\gamma}(\bm h)\le M\}$ and $M = 1/\sqrt{\mathcal S_n(\bm z)}$,
    \item the least squares approximation \eqref{eq:lsqr} using the full and subsampled lattice where the frequency index set $\mathcal B = \{\bm h\in\mathds Z^d : r_{d,\alpha,\bm\gamma}(\bm h)\le M\}$ and $M = 1/(2\sqrt{\mathcal S_n(\bm z)})$ is chosen as in \Cref{koroboverror}, and
    \item the kernel method \eqref{eq:kernelmethod} with both the full and subsampled lattice.
\end{itemize}
With $n$ being the next biggest prime from a power of 2, we compute the generating vector for a lattice using the CBC algorithm from \cite{GS25} rather than the original algorithm from \cite{CKNS20}.
The new algorithm has the advantages that it is much simpler to implement, is extensible to any dimension $d$, and (especially for SPOD weights) is cheaper to run.
Empirically it gives similar $L_2$-approximation error bounds, but equivalent theoretical error bounds are yet to be proven.
For the methods using subsampled points, we generate a random subset of size $|J| = \lceil \sqrt{n}\,\log n\rceil$, which is of the same order as in \Cref{koroboverror} but ignores the involved constants.
The $L_2$-error is then estimated using $50$ random shifts of the initial lattice,
\begin{equation*}
    \int_{\mathds T^d} |(f-A(f))(\bm x)|^2 \;\mathrm d\bm x
    \approx \frac{1}{50n} \sum_{i=1}^{50} \sum_{k=0}^{n-1} |(f-A(f))((\tfrac kn\bm z+\bm\Delta_i)\modone)|^2 \,,
\end{equation*}
where $\bm\Delta_i\in\mathds T^d$ for $i=1,\dots,50$ are chosen independently and uniformly random.
For the kernel method the evaluation at the shifted lattice can be done by the same fast methods as for the original lattice used in the approximation, cf.\ \cite[Section~5]{KKKNS21}.
The least squares approximation can be evaluated at any lattice fast using the \texttt{LFFT}.

The implementation is done in Julia with the \texttt{CG} and \texttt{LSQR} implementation used from \cite{iterativesolvers}.
All the computations were carried out on a high performance computer cluster \cite{katana} with eight parallel threads.

\subsection{Kink function}\label{ssec:kink} 

For a first numerical experiment we choose the target ``kink'' function
\begin{equation}\label{eq:kink}
    f(\bm x) = \Big(\frac{5^{3/4}15}{4\sqrt 3}\Big)^d \prod_{j=1}^{d} \max\Big\{0, \frac 15 - \Big(x_j-\frac 12\Big)^2 \Big\}
\end{equation}
in dimensions $d=2$ and $d=5$.
The 2-dimensional example is depicted in \Cref{fig:2dplots} (left).
This function was considered in \cite{BKUV17, KUV19, BKPU22} and we know that $f\in H_{d,3/2-\varepsilon,\bm\gamma}(\mathds T^d)$ for $\varepsilon>0$ and $\gamma_{\mathfrak u}=1$ for all $\mathfrak u\subseteq\{1,\dots,d\}$.

\begin{figure} \centering
    \hfill
    \begingroup
  \makeatletter
  \providecommand\color[2][]{\GenericError{(gnuplot) \space\space\space\@spaces}{Package color not loaded in conjunction with
      terminal option `colourtext'}{See the gnuplot documentation for explanation.}{Either use 'blacktext' in gnuplot or load the package
      color.sty in LaTeX.}\renewcommand\color[2][]{}}\providecommand\includegraphics[2][]{\GenericError{(gnuplot) \space\space\space\@spaces}{Package graphicx or graphics not loaded}{See the gnuplot documentation for explanation.}{The gnuplot epslatex terminal needs graphicx.sty or graphics.sty.}\renewcommand\includegraphics[2][]{}}\providecommand\rotatebox[2]{#2}\@ifundefined{ifGPcolor}{\newif\ifGPcolor
    \GPcolortrue
  }{}\@ifundefined{ifGPblacktext}{\newif\ifGPblacktext
    \GPblacktexttrue
  }{}\let\gplgaddtomacro\g@addto@macro
\gdef\gplbacktext{}\gdef\gplfronttext{}\makeatother
  \ifGPblacktext
\def\colorrgb#1{}\def\colorgray#1{}\else
\ifGPcolor
      \def\colorrgb#1{\color[rgb]{#1}}\def\colorgray#1{\color[gray]{#1}}\expandafter\def\csname LTw\endcsname{\color{white}}\expandafter\def\csname LTb\endcsname{\color{black}}\expandafter\def\csname LTa\endcsname{\color{black}}\expandafter\def\csname LT0\endcsname{\color[rgb]{1,0,0}}\expandafter\def\csname LT1\endcsname{\color[rgb]{0,1,0}}\expandafter\def\csname LT2\endcsname{\color[rgb]{0,0,1}}\expandafter\def\csname LT3\endcsname{\color[rgb]{1,0,1}}\expandafter\def\csname LT4\endcsname{\color[rgb]{0,1,1}}\expandafter\def\csname LT5\endcsname{\color[rgb]{1,1,0}}\expandafter\def\csname LT6\endcsname{\color[rgb]{0,0,0}}\expandafter\def\csname LT7\endcsname{\color[rgb]{1,0.3,0}}\expandafter\def\csname LT8\endcsname{\color[rgb]{0.5,0.5,0.5}}\else
\def\colorrgb#1{\color{black}}\def\colorgray#1{\color[gray]{#1}}\expandafter\def\csname LTw\endcsname{\color{white}}\expandafter\def\csname LTb\endcsname{\color{black}}\expandafter\def\csname LTa\endcsname{\color{black}}\expandafter\def\csname LT0\endcsname{\color{black}}\expandafter\def\csname LT1\endcsname{\color{black}}\expandafter\def\csname LT2\endcsname{\color{black}}\expandafter\def\csname LT3\endcsname{\color{black}}\expandafter\def\csname LT4\endcsname{\color{black}}\expandafter\def\csname LT5\endcsname{\color{black}}\expandafter\def\csname LT6\endcsname{\color{black}}\expandafter\def\csname LT7\endcsname{\color{black}}\expandafter\def\csname LT8\endcsname{\color{black}}\fi
  \fi
    \setlength{\unitlength}{0.0500bp}\ifx\gptboxheight\undefined \newlength{\gptboxheight}\newlength{\gptboxwidth}\newsavebox{\gptboxtext}\fi \setlength{\fboxrule}{0.5pt}\setlength{\fboxsep}{1pt}\definecolor{tbcol}{rgb}{1,1,1}\begin{picture}(2820.00,2260.00)\gplgaddtomacro\gplbacktext{\csname LTb\endcsname \put(2498,624){\makebox(0,0)[l]{\strut{}\scriptsize $(0,0)$}}\csname LTb\endcsname \put(1669,225){\makebox(0,0)[l]{\strut{}\scriptsize $(0,1)$}}\csname LTb\endcsname \put(213,357){\makebox(0,0)[l]{\strut{}\scriptsize $(1,1)$}}}\gplgaddtomacro\gplfronttext{\csname LTb\endcsname \put(183,728){\makebox(0,0)[r]{\strut{}\scriptsize 0}}\csname LTb\endcsname \put(183,1235){\makebox(0,0)[r]{\strut{}\scriptsize 1}}\csname LTb\endcsname \put(183,1742){\makebox(0,0)[r]{\strut{}\scriptsize 2}}}\gplbacktext
    \put(0,0){\includegraphics[width={141.00bp},height={113.00bp}]{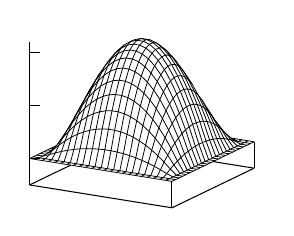}}\gplfronttext
  \end{picture}\endgroup
     \hfill
    \begingroup
  \makeatletter
  \providecommand\color[2][]{\GenericError{(gnuplot) \space\space\space\@spaces}{Package color not loaded in conjunction with
      terminal option `colourtext'}{See the gnuplot documentation for explanation.}{Either use 'blacktext' in gnuplot or load the package
      color.sty in LaTeX.}\renewcommand\color[2][]{}}\providecommand\includegraphics[2][]{\GenericError{(gnuplot) \space\space\space\@spaces}{Package graphicx or graphics not loaded}{See the gnuplot documentation for explanation.}{The gnuplot epslatex terminal needs graphicx.sty or graphics.sty.}\renewcommand\includegraphics[2][]{}}\providecommand\rotatebox[2]{#2}\@ifundefined{ifGPcolor}{\newif\ifGPcolor
    \GPcolortrue
  }{}\@ifundefined{ifGPblacktext}{\newif\ifGPblacktext
    \GPblacktexttrue
  }{}\let\gplgaddtomacro\g@addto@macro
\gdef\gplbacktext{}\gdef\gplfronttext{}\makeatother
  \ifGPblacktext
\def\colorrgb#1{}\def\colorgray#1{}\else
\ifGPcolor
      \def\colorrgb#1{\color[rgb]{#1}}\def\colorgray#1{\color[gray]{#1}}\expandafter\def\csname LTw\endcsname{\color{white}}\expandafter\def\csname LTb\endcsname{\color{black}}\expandafter\def\csname LTa\endcsname{\color{black}}\expandafter\def\csname LT0\endcsname{\color[rgb]{1,0,0}}\expandafter\def\csname LT1\endcsname{\color[rgb]{0,1,0}}\expandafter\def\csname LT2\endcsname{\color[rgb]{0,0,1}}\expandafter\def\csname LT3\endcsname{\color[rgb]{1,0,1}}\expandafter\def\csname LT4\endcsname{\color[rgb]{0,1,1}}\expandafter\def\csname LT5\endcsname{\color[rgb]{1,1,0}}\expandafter\def\csname LT6\endcsname{\color[rgb]{0,0,0}}\expandafter\def\csname LT7\endcsname{\color[rgb]{1,0.3,0}}\expandafter\def\csname LT8\endcsname{\color[rgb]{0.5,0.5,0.5}}\else
\def\colorrgb#1{\color{black}}\def\colorgray#1{\color[gray]{#1}}\expandafter\def\csname LTw\endcsname{\color{white}}\expandafter\def\csname LTb\endcsname{\color{black}}\expandafter\def\csname LTa\endcsname{\color{black}}\expandafter\def\csname LT0\endcsname{\color{black}}\expandafter\def\csname LT1\endcsname{\color{black}}\expandafter\def\csname LT2\endcsname{\color{black}}\expandafter\def\csname LT3\endcsname{\color{black}}\expandafter\def\csname LT4\endcsname{\color{black}}\expandafter\def\csname LT5\endcsname{\color{black}}\expandafter\def\csname LT6\endcsname{\color{black}}\expandafter\def\csname LT7\endcsname{\color{black}}\expandafter\def\csname LT8\endcsname{\color{black}}\fi
  \fi
    \setlength{\unitlength}{0.0500bp}\ifx\gptboxheight\undefined \newlength{\gptboxheight}\newlength{\gptboxwidth}\newsavebox{\gptboxtext}\fi \setlength{\fboxrule}{0.5pt}\setlength{\fboxsep}{1pt}\definecolor{tbcol}{rgb}{1,1,1}\begin{picture}(2820.00,2260.00)\gplgaddtomacro\gplbacktext{\csname LTb\endcsname \put(2498,624){\makebox(0,0)[l]{\strut{}\scriptsize $(0,0)$}}\csname LTb\endcsname \put(1669,225){\makebox(0,0)[l]{\strut{}\scriptsize $(0,1)$}}\csname LTb\endcsname \put(213,357){\makebox(0,0)[l]{\strut{}\scriptsize $(1,1)$}}}\gplgaddtomacro\gplfronttext{\csname LTb\endcsname \put(183,817){\makebox(0,0)[r]{\strut{}\scriptsize 1}}\csname LTb\endcsname \put(183,1501){\makebox(0,0)[r]{\strut{}\scriptsize 2}}}\gplbacktext
    \put(0,0){\includegraphics[width={141.00bp},height={113.00bp}]{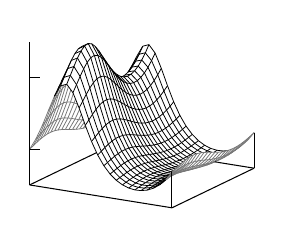}}\gplfronttext
  \end{picture}\endgroup
     \hfill
    \phantom.

    \caption{``Kink'' function \eqref{eq:kink} (left) and ``reciprocal'' function \eqref{eq:reciprocal} with $q=2.5$ (right), both in dimension $d=2$.}\label{fig:2dplots}
\end{figure} 

For the approximation we chose product weights with parameters
\begin{equation*}
    \alpha = 1
    \quad\text{and}\quad
    \gamma_j = \frac 12 \,.
\end{equation*}
The result for dimensions $d=2$ and $d=5$ are depicted in \Cref{fig:kink_E}.

We start by evaluating the 2-dimensional example, hoping to observe the theoretical convergence rate.
All rates displayed in the legend were estimated by fitting the last ten data points.
For the methods using subsampled lattices, we have seen in \Cref{cor} that the logarithm is negligible in the asymptotic regime, but as every computation will be in the preasymptotic regime, we opted for fitting the rate to $|J|/\log|J|$ instead of $|J|$ itself, which is closer to \Cref{koroboverror}.
\begin{itemize}
\item[\scalebox{1.2}{$\bullet$:}]
    For $(n, \sqrt[4]{\mathcal S_n(\bm z)})$ we expect a decay rate arbitrarily close to $1/2$ according to \Cref{Skorobov}.
    The numerics yield a rate of $0.47$ affirming this behavior.

\item[{\color{green}$\blacktriangledown$},{\color{blue}\scalebox{0.7}{\rotatebox{45}{$\blacksquare$}}}:]
    The classical lattice algorithm and the least squares method using the full lattice differ only by a factor of two in the radius of their frequencies.
    According to \Cref{classicalbound} they behave at least as well as $(n, \sqrt[4]{\mathcal S_n(\bm z)})$ with a rate of $1/2$.
    In fact, the same rate of $0.7$ is observed in both methods and they only differ in the corresponding constant.
    The better than expected rate can be explained by these methods ``picking up'' the smoothness.
    Since the smoothness does not affect the shape of the frequency index set, the method for higher smoothness is the same.
    In our case we have smoothness $3/2-\varepsilon$, of which half is ``picked up''.
    That it is only half is intuitively explained by full lattices only achieving half the optimal order of convergence, cf.\ \cite{BKUV17}.

\item[{\color{orange}$\blacktriangle$}:]
    For the kernel method using the full lattice we know by the discussion at the end of \Cref{sec:classical} that it performs at least as well as the classical lattice algorithm with a rate of $1/2$ in the worst-case setting.
    It was further shown in \cite{SK25} that the kernel method (using any set of points) doubles the rate of convergence when the function is doubly as smooth.
    In our case we do not have quite double the smoothness, but still see this effect numerically with a rate of $1.01$.
    Note, that the experiment is only about the approximation error of one individual function which could perform better than the worst-case error.
\item[\scalebox{1.2}{$\circ$:}]
    For the subsampled methods we expect to achieve the same worst-case error but with fewer points.
    With $(|J|/\log |J|, \sqrt[4]{\mathcal S_n(\bm z)})$, we observe a rate of $0.93$, which is about double of $0.47$ from the full lattice $(n,\sqrt[4]{\mathcal S_n(\bm z)})$ (\scalebox{1.2}{$\bullet$}).

\item[{\color{blue}\scalebox{0.7}{\rotatebox{45}{$\square$}}}:]
    According to \Cref{boundwithS} the least squares approximation using the subsampled lattice has at least the same rate as $(|J|/\log |J|, \sqrt[4]{\mathcal S_n(\bm z)})$ (\scalebox{1.2}{$\circ$}), which was $0.93$.
    For reasons explained above, the least squares method is able to ``pick up'' the smoothness, which is $3/2-\varepsilon$.
    Numerically we observe a rate of $1.37$.

\item[{\color{orange}\scalebox{0.7}{$\triangle$}}:]
    The kernel method performs at least as well as any other method using the same points.
    Since the least squares method using the subsampled lattice has at least linear decay, this holds for the kernel method using the subsampled points as well.
    Additionally the doubling of the rate effect from \cite{SK25} applies again.
    Numerically we observe the best rate among all methods with $1.47$.
\end{itemize}

\begin{figure} \centering
    \scalebox{0.9}{
    \begingroup
  \makeatletter
  \providecommand\color[2][]{\GenericError{(gnuplot) \space\space\space\@spaces}{Package color not loaded in conjunction with
      terminal option `colourtext'}{See the gnuplot documentation for explanation.}{Either use 'blacktext' in gnuplot or load the package
      color.sty in LaTeX.}\renewcommand\color[2][]{}}\providecommand\includegraphics[2][]{\GenericError{(gnuplot) \space\space\space\@spaces}{Package graphicx or graphics not loaded}{See the gnuplot documentation for explanation.}{The gnuplot epslatex terminal needs graphicx.sty or graphics.sty.}\renewcommand\includegraphics[2][]{}}\providecommand\rotatebox[2]{#2}\@ifundefined{ifGPcolor}{\newif\ifGPcolor
    \GPcolortrue
  }{}\@ifundefined{ifGPblacktext}{\newif\ifGPblacktext
    \GPblacktexttrue
  }{}\let\gplgaddtomacro\g@addto@macro
\gdef\gplbacktext{}\gdef\gplfronttext{}\makeatother
  \ifGPblacktext
\def\colorrgb#1{}\def\colorgray#1{}\else
\ifGPcolor
      \def\colorrgb#1{\color[rgb]{#1}}\def\colorgray#1{\color[gray]{#1}}\expandafter\def\csname LTw\endcsname{\color{white}}\expandafter\def\csname LTb\endcsname{\color{black}}\expandafter\def\csname LTa\endcsname{\color{black}}\expandafter\def\csname LT0\endcsname{\color[rgb]{1,0,0}}\expandafter\def\csname LT1\endcsname{\color[rgb]{0,1,0}}\expandafter\def\csname LT2\endcsname{\color[rgb]{0,0,1}}\expandafter\def\csname LT3\endcsname{\color[rgb]{1,0,1}}\expandafter\def\csname LT4\endcsname{\color[rgb]{0,1,1}}\expandafter\def\csname LT5\endcsname{\color[rgb]{1,1,0}}\expandafter\def\csname LT6\endcsname{\color[rgb]{0,0,0}}\expandafter\def\csname LT7\endcsname{\color[rgb]{1,0.3,0}}\expandafter\def\csname LT8\endcsname{\color[rgb]{0.5,0.5,0.5}}\else
\def\colorrgb#1{\color{black}}\def\colorgray#1{\color[gray]{#1}}\expandafter\def\csname LTw\endcsname{\color{white}}\expandafter\def\csname LTb\endcsname{\color{black}}\expandafter\def\csname LTa\endcsname{\color{black}}\expandafter\def\csname LT0\endcsname{\color{black}}\expandafter\def\csname LT1\endcsname{\color{black}}\expandafter\def\csname LT2\endcsname{\color{black}}\expandafter\def\csname LT3\endcsname{\color{black}}\expandafter\def\csname LT4\endcsname{\color{black}}\expandafter\def\csname LT5\endcsname{\color{black}}\expandafter\def\csname LT6\endcsname{\color{black}}\expandafter\def\csname LT7\endcsname{\color{black}}\expandafter\def\csname LT8\endcsname{\color{black}}\fi
  \fi
    \setlength{\unitlength}{0.0500bp}\ifx\gptboxheight\undefined \newlength{\gptboxheight}\newlength{\gptboxwidth}\newsavebox{\gptboxtext}\fi \setlength{\fboxrule}{0.5pt}\setlength{\fboxsep}{1pt}\definecolor{tbcol}{rgb}{1,1,1}\begin{picture}(7920.00,3960.00)\gplgaddtomacro\gplbacktext{\colorrgb{0.00,0.00,0.00}\put(251,551){\makebox(0,0)[r]{\strut{}\scriptsize $10^{-6}$}}\colorrgb{0.00,0.00,0.00}\put(251,1076){\makebox(0,0)[r]{\strut{}\scriptsize $10^{-5}$}}\colorrgb{0.00,0.00,0.00}\put(251,1600){\makebox(0,0)[r]{\strut{}\scriptsize $10^{-4}$}}\colorrgb{0.00,0.00,0.00}\put(251,2125){\makebox(0,0)[r]{\strut{}\scriptsize $10^{-3}$}}\colorrgb{0.00,0.00,0.00}\put(251,2650){\makebox(0,0)[r]{\strut{}\scriptsize $10^{-2}$}}\colorrgb{0.00,0.00,0.00}\put(251,3175){\makebox(0,0)[r]{\strut{}\scriptsize $10^{-1}$}}\colorrgb{0.00,0.00,0.00}\put(251,3700){\makebox(0,0)[r]{\strut{}\scriptsize $10^{0}$}}\colorrgb{0.00,0.00,0.00}\put(352,383){\makebox(0,0){\strut{}\scriptsize $10^{2}$}}\colorrgb{0.00,0.00,0.00}\put(1723,383){\makebox(0,0){\strut{}\scriptsize $10^{4}$}}\colorrgb{0.00,0.00,0.00}\put(3094,383){\makebox(0,0){\strut{}\scriptsize $10^{6}$}}\colorrgb{0.00,0.00,0.00}\put(4466,383){\makebox(0,0){\strut{}\scriptsize $10^{8}$}}\csname LTb\endcsname \put(3555,3267){\makebox(0,0)[l]{\strut{}\scriptsize 1}}\csname LTb\endcsname \put(4494,2843){\makebox(0,0)[l]{\strut{}\scriptsize $\tfrac 12$}}\csname LTb\endcsname \put(1436,825){\makebox(0,0){\strut{}\scriptsize 1}}\csname LTb\endcsname \put(736,1442){\makebox(0,0)[r]{\strut{}\scriptsize 1}}\csname LTb\endcsname \put(7188,1063){\makebox(0,0)[l]{\strut{}\scriptsize $\alpha = 1$}}}\gplgaddtomacro\gplfronttext{\csname LTb\endcsname \put(5382,1363){\makebox(0,0)[l]{\strut{}\scriptsize\parbox{85pt}{$(|J|, \|f-A_{\text{ker}}^{\bm X_J} f\|_{L_2})$}$\Big(\frac{|J|}{\log|J|}\Big)^{-1.47}$}}\csname LTb\endcsname \put(5382,1722){\makebox(0,0)[l]{\strut{}\scriptsize\parbox{85pt}{$(|J|, \|f-S_{\mathcal B}^{\bm X_J} f\|_{L_2})$}$\Big(\frac{|J|}{\log|J|}\Big)^{-1.37}$}}\csname LTb\endcsname \put(5382,2082){\makebox(0,0)[l]{\strut{}\scriptsize\parbox{85pt}{$(|J|, \sqrt[4]{\mathcal S_n(\bm z)})$}$\Big(\frac{|J|}{\log|J|}\Big)^{-0.93}$}}\csname LTb\endcsname \put(5382,2441){\makebox(0,0)[l]{\strut{}\scriptsize\parbox{85pt}{$(n, \|f-A_{\text{ker}}^{\bm X} f\|_{L_2})$}$n^{-1.01}$}}\csname LTb\endcsname \put(5382,2801){\makebox(0,0)[l]{\strut{}\scriptsize\parbox{85pt}{$(n, \|f-S_{\mathcal B}^{\bm X} f\|_{L_2})$}$n^{-0.70}$}}\csname LTb\endcsname \put(5382,3160){\makebox(0,0)[l]{\strut{}\scriptsize\parbox{85pt}{$(n, \|f-A_{\mathcal A}^{\bm X} f\|_{L_2})$}$n^{-0.70}$}}\csname LTb\endcsname \put(5382,3520){\makebox(0,0)[l]{\strut{}\scriptsize\parbox{85pt}{$(n, \sqrt[4]{\mathcal S_n(\bm z)})$}$n^{-0.47}$}}\csname LTb\endcsname \put(2615,119){\makebox(0,0){\strut{}\scriptsize number of points}}\csname LTb\endcsname \put(2615,3939){\makebox(0,0){\strut{}kink function $d=2$}}}\gplbacktext
    \put(0,0){\includegraphics[width={396.00bp},height={198.00bp}]{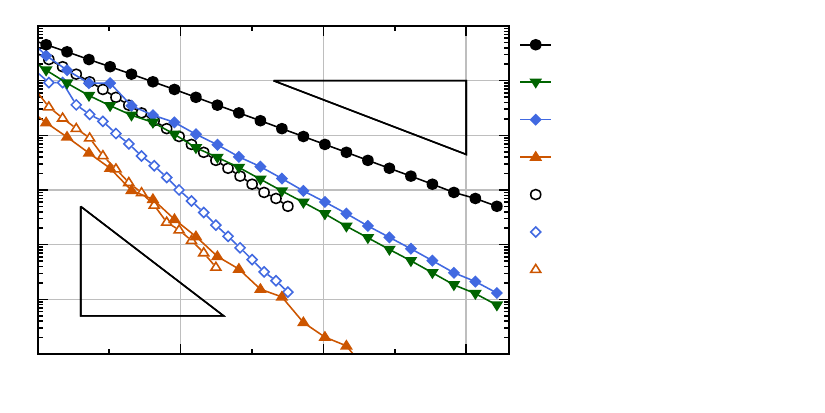}}\gplfronttext
  \end{picture}\endgroup
     }
    
    \vspace{20pt}

    \scalebox{0.9}{
    \begingroup
  \makeatletter
  \providecommand\color[2][]{\GenericError{(gnuplot) \space\space\space\@spaces}{Package color not loaded in conjunction with
      terminal option `colourtext'}{See the gnuplot documentation for explanation.}{Either use 'blacktext' in gnuplot or load the package
      color.sty in LaTeX.}\renewcommand\color[2][]{}}\providecommand\includegraphics[2][]{\GenericError{(gnuplot) \space\space\space\@spaces}{Package graphicx or graphics not loaded}{See the gnuplot documentation for explanation.}{The gnuplot epslatex terminal needs graphicx.sty or graphics.sty.}\renewcommand\includegraphics[2][]{}}\providecommand\rotatebox[2]{#2}\@ifundefined{ifGPcolor}{\newif\ifGPcolor
    \GPcolortrue
  }{}\@ifundefined{ifGPblacktext}{\newif\ifGPblacktext
    \GPblacktexttrue
  }{}\let\gplgaddtomacro\g@addto@macro
\gdef\gplbacktext{}\gdef\gplfronttext{}\makeatother
  \ifGPblacktext
\def\colorrgb#1{}\def\colorgray#1{}\else
\ifGPcolor
      \def\colorrgb#1{\color[rgb]{#1}}\def\colorgray#1{\color[gray]{#1}}\expandafter\def\csname LTw\endcsname{\color{white}}\expandafter\def\csname LTb\endcsname{\color{black}}\expandafter\def\csname LTa\endcsname{\color{black}}\expandafter\def\csname LT0\endcsname{\color[rgb]{1,0,0}}\expandafter\def\csname LT1\endcsname{\color[rgb]{0,1,0}}\expandafter\def\csname LT2\endcsname{\color[rgb]{0,0,1}}\expandafter\def\csname LT3\endcsname{\color[rgb]{1,0,1}}\expandafter\def\csname LT4\endcsname{\color[rgb]{0,1,1}}\expandafter\def\csname LT5\endcsname{\color[rgb]{1,1,0}}\expandafter\def\csname LT6\endcsname{\color[rgb]{0,0,0}}\expandafter\def\csname LT7\endcsname{\color[rgb]{1,0.3,0}}\expandafter\def\csname LT8\endcsname{\color[rgb]{0.5,0.5,0.5}}\else
\def\colorrgb#1{\color{black}}\def\colorgray#1{\color[gray]{#1}}\expandafter\def\csname LTw\endcsname{\color{white}}\expandafter\def\csname LTb\endcsname{\color{black}}\expandafter\def\csname LTa\endcsname{\color{black}}\expandafter\def\csname LT0\endcsname{\color{black}}\expandafter\def\csname LT1\endcsname{\color{black}}\expandafter\def\csname LT2\endcsname{\color{black}}\expandafter\def\csname LT3\endcsname{\color{black}}\expandafter\def\csname LT4\endcsname{\color{black}}\expandafter\def\csname LT5\endcsname{\color{black}}\expandafter\def\csname LT6\endcsname{\color{black}}\expandafter\def\csname LT7\endcsname{\color{black}}\expandafter\def\csname LT8\endcsname{\color{black}}\fi
  \fi
    \setlength{\unitlength}{0.0500bp}\ifx\gptboxheight\undefined \newlength{\gptboxheight}\newlength{\gptboxwidth}\newsavebox{\gptboxtext}\fi \setlength{\fboxrule}{0.5pt}\setlength{\fboxsep}{1pt}\definecolor{tbcol}{rgb}{1,1,1}\begin{picture}(7920.00,3960.00)\gplgaddtomacro\gplbacktext{\colorrgb{0.00,0.00,0.00}\put(251,551){\makebox(0,0)[r]{\strut{}\scriptsize $10^{-4}$}}\colorrgb{0.00,0.00,0.00}\put(251,1181){\makebox(0,0)[r]{\strut{}\scriptsize $10^{-3}$}}\colorrgb{0.00,0.00,0.00}\put(251,1810){\makebox(0,0)[r]{\strut{}\scriptsize $10^{-2}$}}\colorrgb{0.00,0.00,0.00}\put(251,2440){\makebox(0,0)[r]{\strut{}\scriptsize $10^{-1}$}}\colorrgb{0.00,0.00,0.00}\put(251,3070){\makebox(0,0)[r]{\strut{}\scriptsize $10^{0}$}}\colorrgb{0.00,0.00,0.00}\put(251,3700){\makebox(0,0)[r]{\strut{}\scriptsize $10^{1}$}}\colorrgb{0.00,0.00,0.00}\put(352,383){\makebox(0,0){\strut{}\scriptsize $10^{2}$}}\colorrgb{0.00,0.00,0.00}\put(1675,383){\makebox(0,0){\strut{}\scriptsize $10^{4}$}}\colorrgb{0.00,0.00,0.00}\put(2997,383){\makebox(0,0){\strut{}\scriptsize $10^{6}$}}\colorrgb{0.00,0.00,0.00}\put(4320,383){\makebox(0,0){\strut{}\scriptsize $10^{8}$}}\csname LTb\endcsname \put(3368,3273){\makebox(0,0)[l]{\strut{}\scriptsize 1}}\csname LTb\endcsname \put(4200,2817){\makebox(0,0)[l]{\strut{}\scriptsize $\tfrac 12$}}\csname LTb\endcsname \put(1298,919){\makebox(0,0){\strut{}\scriptsize 1}}\csname LTb\endcsname \put(723,1526){\makebox(0,0)[r]{\strut{}\scriptsize 1}}\csname LTb\endcsname \put(7188,1063){\makebox(0,0)[l]{\strut{}\scriptsize $\alpha = 1$}}}\gplgaddtomacro\gplfronttext{\csname LTb\endcsname \put(5382,1363){\makebox(0,0)[l]{\strut{}\scriptsize\parbox{85pt}{$(|J|, \|f-A_{\text{ker}}^{\bm X_J} f\|_{L_2})$}$\Big(\frac{|J|}{\log|J|}\Big)^{-0.93}$}}\csname LTb\endcsname \put(5382,1722){\makebox(0,0)[l]{\strut{}\scriptsize\parbox{85pt}{$(|J|, \|f-S_{\mathcal B}^{\bm X_J} f\|_{L_2})$}$\Big(\frac{|J|}{\log|J|}\Big)^{-1.03}$}}\csname LTb\endcsname \put(5382,2082){\makebox(0,0)[l]{\strut{}\scriptsize\parbox{85pt}{$(|J|, \sqrt[4]{\mathcal S_n(\bm z)})$}$\Big(\frac{|J|}{\log|J|}\Big)^{-0.77}$}}\csname LTb\endcsname \put(5382,2441){\makebox(0,0)[l]{\strut{}\scriptsize\parbox{85pt}{$(n, \|f-A_{\text{ker}}^{\bm X} f\|_{L_2})$}$n^{-0.81}$}}\csname LTb\endcsname \put(5382,2801){\makebox(0,0)[l]{\strut{}\scriptsize\parbox{85pt}{$(n, \|f-S_{\mathcal B}^{\bm X} f\|_{L_2})$}$n^{-0.53}$}}\csname LTb\endcsname \put(5382,3160){\makebox(0,0)[l]{\strut{}\scriptsize\parbox{85pt}{$(n, \|f-A_{\mathcal A}^{\bm X} f\|_{L_2})$}$n^{-0.52}$}}\csname LTb\endcsname \put(5382,3520){\makebox(0,0)[l]{\strut{}\scriptsize\parbox{85pt}{$(n, \sqrt[4]{\mathcal S_n(\bm z)})$}$n^{-0.39}$}}\csname LTb\endcsname \put(2615,119){\makebox(0,0){\strut{}\scriptsize number of points}}\csname LTb\endcsname \put(2615,3939){\makebox(0,0){\strut{}kink function $d=5$}}}\gplbacktext
    \put(0,0){\includegraphics[width={396.00bp},height={198.00bp}]{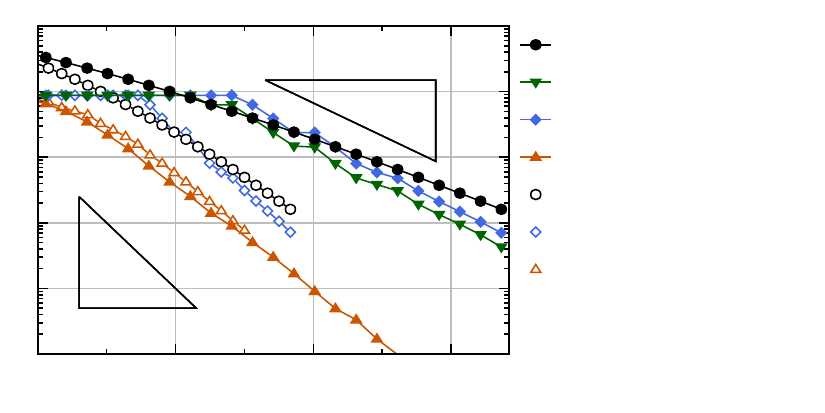}}\gplfronttext
  \end{picture}\endgroup
     }
    \caption{$L_2$-error of the approximations of the kink function, with $d=2$ (top) and $d=5$ (bottom) with respect to the number of sampling points for the classical lattice algorithm, the least squares method, and the kernel method using the full and subsampled lattice.}\label{fig:kink_E}
\end{figure} 

The 5-dimensional example shows the same qualitative but not quantitative behavior, as increasing the dimension makes the problem a lot harder when all dimensions are equally important.
The dimension enters logarithmically as we have for the ``linear width'' $a_n \sim n^{-\alpha}(\log n)^{(d-1)\alpha}$, which is a lower bound to the worst-case error using $n$ samples, cf.\ \cite{DTU18}.
One notable difference is a staircase-like error behavior for the classical lattice algorithm as well as the least squares approximation.
This is due to the number of frequencies not being continuous with respect to the radius of the frequency index set.
Especially for small radii a bunch of frequencies are added at once relative to the total number of frequencies, resulting in a jump.

\subsection{Reciprocal function} 

For a second batch of numerical experiments we choose the $d=100$ dimensional target ``reciprocal'' function
\begin{equation}\label{eq:reciprocal}
    f(\bm x) = \frac{1}{a(\bm x)} \,,\quad
    a(\bm x) = 1+0.5 \sum_{j=1}^{d} j^{-q} \sin(2\pi x_j)\,,
\end{equation}
with decay parameter $q=6$ and $q=2.5$.
The function in two dimensions is depicted in \Cref{fig:2dplots} (right).
This function was considered in e.g., \cite[Section~4.2]{KKNS25} and solves the algebraic equation $a(\bm x)f(\bm x)=1$, mimicking the features of a partial differential equation with a random coefficient whilst avoiding the complexity of a spatial variable or the need of a finite element solver.

For the approximation we choose POD weights with parameters
\begin{equation*}
    \alpha \in \{1,2\}
    \quad\text{and}\quad
    \gamma_{\mathfrak u} = |\mathfrak u|! \prod_{j\in \mathfrak u}j^{-q} \,.
\end{equation*}

Following the same procedure as in \Cref{ssec:kink},
in \Cref{fig:reciprocal6_E} we depict the results for $q=6$, with $\alpha=1$ (top) and $\alpha=2$ (bottom).
\begin{figure} \centering
    \scalebox{0.9}{
    \begingroup
  \makeatletter
  \providecommand\color[2][]{\GenericError{(gnuplot) \space\space\space\@spaces}{Package color not loaded in conjunction with
      terminal option `colourtext'}{See the gnuplot documentation for explanation.}{Either use 'blacktext' in gnuplot or load the package
      color.sty in LaTeX.}\renewcommand\color[2][]{}}\providecommand\includegraphics[2][]{\GenericError{(gnuplot) \space\space\space\@spaces}{Package graphicx or graphics not loaded}{See the gnuplot documentation for explanation.}{The gnuplot epslatex terminal needs graphicx.sty or graphics.sty.}\renewcommand\includegraphics[2][]{}}\providecommand\rotatebox[2]{#2}\@ifundefined{ifGPcolor}{\newif\ifGPcolor
    \GPcolortrue
  }{}\@ifundefined{ifGPblacktext}{\newif\ifGPblacktext
    \GPblacktexttrue
  }{}\let\gplgaddtomacro\g@addto@macro
\gdef\gplbacktext{}\gdef\gplfronttext{}\makeatother
  \ifGPblacktext
\def\colorrgb#1{}\def\colorgray#1{}\else
\ifGPcolor
      \def\colorrgb#1{\color[rgb]{#1}}\def\colorgray#1{\color[gray]{#1}}\expandafter\def\csname LTw\endcsname{\color{white}}\expandafter\def\csname LTb\endcsname{\color{black}}\expandafter\def\csname LTa\endcsname{\color{black}}\expandafter\def\csname LT0\endcsname{\color[rgb]{1,0,0}}\expandafter\def\csname LT1\endcsname{\color[rgb]{0,1,0}}\expandafter\def\csname LT2\endcsname{\color[rgb]{0,0,1}}\expandafter\def\csname LT3\endcsname{\color[rgb]{1,0,1}}\expandafter\def\csname LT4\endcsname{\color[rgb]{0,1,1}}\expandafter\def\csname LT5\endcsname{\color[rgb]{1,1,0}}\expandafter\def\csname LT6\endcsname{\color[rgb]{0,0,0}}\expandafter\def\csname LT7\endcsname{\color[rgb]{1,0.3,0}}\expandafter\def\csname LT8\endcsname{\color[rgb]{0.5,0.5,0.5}}\else
\def\colorrgb#1{\color{black}}\def\colorgray#1{\color[gray]{#1}}\expandafter\def\csname LTw\endcsname{\color{white}}\expandafter\def\csname LTb\endcsname{\color{black}}\expandafter\def\csname LTa\endcsname{\color{black}}\expandafter\def\csname LT0\endcsname{\color{black}}\expandafter\def\csname LT1\endcsname{\color{black}}\expandafter\def\csname LT2\endcsname{\color{black}}\expandafter\def\csname LT3\endcsname{\color{black}}\expandafter\def\csname LT4\endcsname{\color{black}}\expandafter\def\csname LT5\endcsname{\color{black}}\expandafter\def\csname LT6\endcsname{\color{black}}\expandafter\def\csname LT7\endcsname{\color{black}}\expandafter\def\csname LT8\endcsname{\color{black}}\fi
  \fi
    \setlength{\unitlength}{0.0500bp}\ifx\gptboxheight\undefined \newlength{\gptboxheight}\newlength{\gptboxwidth}\newsavebox{\gptboxtext}\fi \setlength{\fboxrule}{0.5pt}\setlength{\fboxsep}{1pt}\definecolor{tbcol}{rgb}{1,1,1}\begin{picture}(7920.00,3960.00)\gplgaddtomacro\gplbacktext{\colorrgb{0.00,0.00,0.00}\put(251,901){\makebox(0,0)[r]{\strut{}\scriptsize $10^{-8}$}}\colorrgb{0.00,0.00,0.00}\put(251,1251){\makebox(0,0)[r]{\strut{}\scriptsize $10^{-7}$}}\colorrgb{0.00,0.00,0.00}\put(251,1600){\makebox(0,0)[r]{\strut{}\scriptsize $10^{-6}$}}\colorrgb{0.00,0.00,0.00}\put(251,1950){\makebox(0,0)[r]{\strut{}\scriptsize $10^{-5}$}}\colorrgb{0.00,0.00,0.00}\put(251,2300){\makebox(0,0)[r]{\strut{}\scriptsize $10^{-4}$}}\colorrgb{0.00,0.00,0.00}\put(251,2650){\makebox(0,0)[r]{\strut{}\scriptsize $10^{-3}$}}\colorrgb{0.00,0.00,0.00}\put(251,3000){\makebox(0,0)[r]{\strut{}\scriptsize $10^{-2}$}}\colorrgb{0.00,0.00,0.00}\put(251,3350){\makebox(0,0)[r]{\strut{}\scriptsize $10^{-1}$}}\colorrgb{0.00,0.00,0.00}\put(251,3700){\makebox(0,0)[r]{\strut{}\scriptsize $10^{0}$}}\colorrgb{0.00,0.00,0.00}\put(352,383){\makebox(0,0){\strut{}\scriptsize $10^{2}$}}\colorrgb{0.00,0.00,0.00}\put(1940,383){\makebox(0,0){\strut{}\scriptsize $10^{4}$}}\colorrgb{0.00,0.00,0.00}\put(3529,383){\makebox(0,0){\strut{}\scriptsize $10^{6}$}}\csname LTb\endcsname \put(3584,3455){\makebox(0,0)[l]{\strut{}\scriptsize 1}}\csname LTb\endcsname \put(4672,3128){\makebox(0,0)[l]{\strut{}\scriptsize $\tfrac 12$}}\csname LTb\endcsname \put(1368,1040){\makebox(0,0){\strut{}\scriptsize 1}}\csname LTb\endcsname \put(558,1495){\makebox(0,0)[r]{\strut{}\scriptsize 1}}\csname LTb\endcsname \put(7188,1063){\makebox(0,0)[l]{\strut{}\scriptsize $\alpha = 1$}}}\gplgaddtomacro\gplfronttext{\csname LTb\endcsname \put(5382,1363){\makebox(0,0)[l]{\strut{}\scriptsize\parbox{85pt}{$(|J|, \|f-A_{\text{ker}}^{\bm X_J} f\|_{L_2})$}$\Big(\frac{|J|}{\log|J|}\Big)^{-1.89}$}}\csname LTb\endcsname \put(5382,1722){\makebox(0,0)[l]{\strut{}\scriptsize\parbox{85pt}{$(|J|, \|f-S_{\mathcal B}^{\bm X_J} f\|_{L_2})$}$\Big(\frac{|J|}{\log|J|}\Big)^{-1.65}$}}\csname LTb\endcsname \put(5382,2082){\makebox(0,0)[l]{\strut{}\scriptsize\parbox{85pt}{$(|J|, \sqrt[4]{\mathcal S_n(\bm z)})$}$\Big(\frac{|J|}{\log|J|}\Big)^{-0.91}$}}\csname LTb\endcsname \put(5382,2441){\makebox(0,0)[l]{\strut{}\scriptsize\parbox{85pt}{$(n, \|f-A_{\text{ker}}^{\bm X} f\|_{L_2})$}$n^{-1.41}$}}\csname LTb\endcsname \put(5382,2801){\makebox(0,0)[l]{\strut{}\scriptsize\parbox{85pt}{$(n, \|f-S_{\mathcal B}^{\bm X} f\|_{L_2})$}$n^{-0.84}$}}\csname LTb\endcsname \put(5382,3160){\makebox(0,0)[l]{\strut{}\scriptsize\parbox{85pt}{$(n, \|f-A_{\mathcal A}^{\bm X} f\|_{L_2})$}$n^{-0.86}$}}\csname LTb\endcsname \put(5382,3520){\makebox(0,0)[l]{\strut{}\scriptsize\parbox{85pt}{$(n, \sqrt[4]{\mathcal S_n(\bm z)})$}$n^{-0.47}$}}\csname LTb\endcsname \put(2615,119){\makebox(0,0){\strut{}\scriptsize number of points}}\csname LTb\endcsname \put(2615,3939){\makebox(0,0){\strut{}reciprocal function $d=100$, $q=6$}}}\gplbacktext
    \put(0,0){\includegraphics[width={396.00bp},height={198.00bp}]{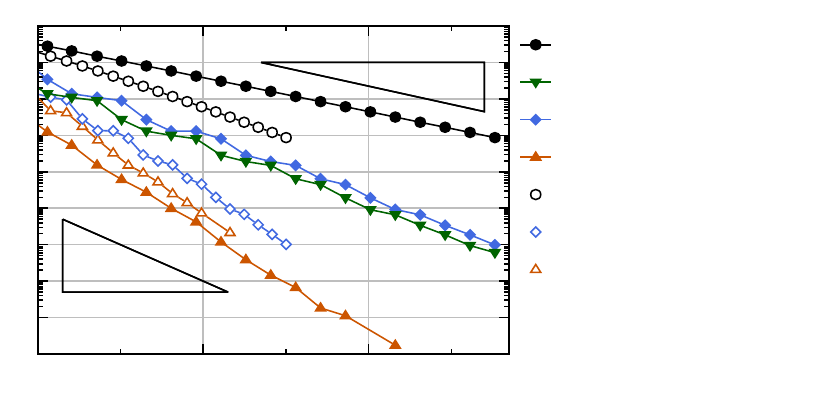}}\gplfronttext
  \end{picture}\endgroup
     }
    
    \vspace{20pt}

    \scalebox{0.9}{
    \begingroup
  \makeatletter
  \providecommand\color[2][]{\GenericError{(gnuplot) \space\space\space\@spaces}{Package color not loaded in conjunction with
      terminal option `colourtext'}{See the gnuplot documentation for explanation.}{Either use 'blacktext' in gnuplot or load the package
      color.sty in LaTeX.}\renewcommand\color[2][]{}}\providecommand\includegraphics[2][]{\GenericError{(gnuplot) \space\space\space\@spaces}{Package graphicx or graphics not loaded}{See the gnuplot documentation for explanation.}{The gnuplot epslatex terminal needs graphicx.sty or graphics.sty.}\renewcommand\includegraphics[2][]{}}\providecommand\rotatebox[2]{#2}\@ifundefined{ifGPcolor}{\newif\ifGPcolor
    \GPcolortrue
  }{}\@ifundefined{ifGPblacktext}{\newif\ifGPblacktext
    \GPblacktexttrue
  }{}\let\gplgaddtomacro\g@addto@macro
\gdef\gplbacktext{}\gdef\gplfronttext{}\makeatother
  \ifGPblacktext
\def\colorrgb#1{}\def\colorgray#1{}\else
\ifGPcolor
      \def\colorrgb#1{\color[rgb]{#1}}\def\colorgray#1{\color[gray]{#1}}\expandafter\def\csname LTw\endcsname{\color{white}}\expandafter\def\csname LTb\endcsname{\color{black}}\expandafter\def\csname LTa\endcsname{\color{black}}\expandafter\def\csname LT0\endcsname{\color[rgb]{1,0,0}}\expandafter\def\csname LT1\endcsname{\color[rgb]{0,1,0}}\expandafter\def\csname LT2\endcsname{\color[rgb]{0,0,1}}\expandafter\def\csname LT3\endcsname{\color[rgb]{1,0,1}}\expandafter\def\csname LT4\endcsname{\color[rgb]{0,1,1}}\expandafter\def\csname LT5\endcsname{\color[rgb]{1,1,0}}\expandafter\def\csname LT6\endcsname{\color[rgb]{0,0,0}}\expandafter\def\csname LT7\endcsname{\color[rgb]{1,0.3,0}}\expandafter\def\csname LT8\endcsname{\color[rgb]{0.5,0.5,0.5}}\else
\def\colorrgb#1{\color{black}}\def\colorgray#1{\color[gray]{#1}}\expandafter\def\csname LTw\endcsname{\color{white}}\expandafter\def\csname LTb\endcsname{\color{black}}\expandafter\def\csname LTa\endcsname{\color{black}}\expandafter\def\csname LT0\endcsname{\color{black}}\expandafter\def\csname LT1\endcsname{\color{black}}\expandafter\def\csname LT2\endcsname{\color{black}}\expandafter\def\csname LT3\endcsname{\color{black}}\expandafter\def\csname LT4\endcsname{\color{black}}\expandafter\def\csname LT5\endcsname{\color{black}}\expandafter\def\csname LT6\endcsname{\color{black}}\expandafter\def\csname LT7\endcsname{\color{black}}\expandafter\def\csname LT8\endcsname{\color{black}}\fi
  \fi
    \setlength{\unitlength}{0.0500bp}\ifx\gptboxheight\undefined \newlength{\gptboxheight}\newlength{\gptboxwidth}\newsavebox{\gptboxtext}\fi \setlength{\fboxrule}{0.5pt}\setlength{\fboxsep}{1pt}\definecolor{tbcol}{rgb}{1,1,1}\begin{picture}(7920.00,3960.00)\gplgaddtomacro\gplbacktext{\colorrgb{0.00,0.00,0.00}\put(251,944){\makebox(0,0)[r]{\strut{}\scriptsize $10^{-8}$}}\colorrgb{0.00,0.00,0.00}\put(251,1338){\makebox(0,0)[r]{\strut{}\scriptsize $10^{-7}$}}\colorrgb{0.00,0.00,0.00}\put(251,1732){\makebox(0,0)[r]{\strut{}\scriptsize $10^{-6}$}}\colorrgb{0.00,0.00,0.00}\put(251,2125){\makebox(0,0)[r]{\strut{}\scriptsize $10^{-5}$}}\colorrgb{0.00,0.00,0.00}\put(251,2519){\makebox(0,0)[r]{\strut{}\scriptsize $10^{-4}$}}\colorrgb{0.00,0.00,0.00}\put(251,2913){\makebox(0,0)[r]{\strut{}\scriptsize $10^{-3}$}}\colorrgb{0.00,0.00,0.00}\put(251,3306){\makebox(0,0)[r]{\strut{}\scriptsize $10^{-2}$}}\colorrgb{0.00,0.00,0.00}\put(251,3700){\makebox(0,0)[r]{\strut{}\scriptsize $10^{-1}$}}\colorrgb{0.00,0.00,0.00}\put(352,383){\makebox(0,0){\strut{}\scriptsize $10^{2}$}}\colorrgb{0.00,0.00,0.00}\put(1609,383){\makebox(0,0){\strut{}\scriptsize $10^{3}$}}\colorrgb{0.00,0.00,0.00}\put(2865,383){\makebox(0,0){\strut{}\scriptsize $10^{4}$}}\colorrgb{0.00,0.00,0.00}\put(4122,383){\makebox(0,0){\strut{}\scriptsize $10^{5}$}}\csname LTb\endcsname \put(3581,3543){\makebox(0,0)[l]{\strut{}\scriptsize 1}}\csname LTb\endcsname \put(4674,3107){\makebox(0,0)[l]{\strut{}\scriptsize 1}}\csname LTb\endcsname \put(1632,826){\makebox(0,0){\strut{}\scriptsize 1}}\csname LTb\endcsname \put(678,1526){\makebox(0,0)[r]{\strut{}\scriptsize 2}}\csname LTb\endcsname \put(7188,1063){\makebox(0,0)[l]{\strut{}\scriptsize $\alpha = 2$}}}\gplgaddtomacro\gplfronttext{\csname LTb\endcsname \put(5382,1363){\makebox(0,0)[l]{\strut{}\scriptsize\parbox{85pt}{$(|J|, \|f-A_{\text{ker}}^{\bm X_J} f\|_{L_2})$}$\Big(\frac{|J|}{\log|J|}\Big)^{-2.80}$}}\csname LTb\endcsname \put(5382,1722){\makebox(0,0)[l]{\strut{}\scriptsize\parbox{85pt}{$(|J|, \|f-S_{\mathcal B}^{\bm X_J} f\|_{L_2})$}$\Big(\frac{|J|}{\log|J|}\Big)^{-2.79}$}}\csname LTb\endcsname \put(5382,2082){\makebox(0,0)[l]{\strut{}\scriptsize\parbox{85pt}{$(|J|, \sqrt[4]{\mathcal S_n(\bm z)})$}$\Big(\frac{|J|}{\log|J|}\Big)^{-1.60}$}}\csname LTb\endcsname \put(5382,2441){\makebox(0,0)[l]{\strut{}\scriptsize\parbox{85pt}{$(n, \|f-A_{\text{ker}}^{\bm X} f\|_{L_2})$}$n^{-1.94}$}}\csname LTb\endcsname \put(5382,2801){\makebox(0,0)[l]{\strut{}\scriptsize\parbox{85pt}{$(n, \|f-S_{\mathcal B}^{\bm X} f\|_{L_2})$}$n^{-1.50}$}}\csname LTb\endcsname \put(5382,3160){\makebox(0,0)[l]{\strut{}\scriptsize\parbox{85pt}{$(n, \|f-A_{\mathcal A}^{\bm X} f\|_{L_2})$}$n^{-1.45}$}}\csname LTb\endcsname \put(5382,3520){\makebox(0,0)[l]{\strut{}\scriptsize\parbox{85pt}{$(n, \sqrt[4]{\mathcal S_n(\bm z)})$}$n^{-0.83}$}}\csname LTb\endcsname \put(2615,119){\makebox(0,0){\strut{}\scriptsize number of points}}\csname LTb\endcsname \put(2615,3939){\makebox(0,0){\strut{}reciprocal function $d=100$, $q=6$}}}\gplbacktext
    \put(0,0){\includegraphics[width={396.00bp},height={198.00bp}]{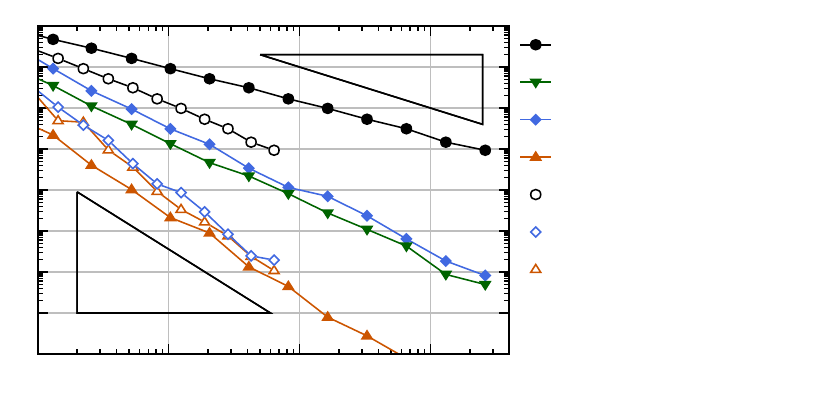}}\gplfronttext
  \end{picture}\endgroup
     }
    \caption{$L_2$-error of the approximations of the reciprocal function in $d=100$ dimensions and decay parameter $q=6$, with smoothness parameter $\alpha=1$ (top) and $\alpha=2$ (bottom), with respect to the number of sampling points for the classical lattice algorithm, the least squares method, and the kernel method using the full and subsampled lattice.}\label{fig:reciprocal6_E}
\end{figure} \begin{figure} \centering
    \scalebox{0.9}{
    \begingroup
  \makeatletter
  \providecommand\color[2][]{\GenericError{(gnuplot) \space\space\space\@spaces}{Package color not loaded in conjunction with
      terminal option `colourtext'}{See the gnuplot documentation for explanation.}{Either use 'blacktext' in gnuplot or load the package
      color.sty in LaTeX.}\renewcommand\color[2][]{}}\providecommand\includegraphics[2][]{\GenericError{(gnuplot) \space\space\space\@spaces}{Package graphicx or graphics not loaded}{See the gnuplot documentation for explanation.}{The gnuplot epslatex terminal needs graphicx.sty or graphics.sty.}\renewcommand\includegraphics[2][]{}}\providecommand\rotatebox[2]{#2}\@ifundefined{ifGPcolor}{\newif\ifGPcolor
    \GPcolortrue
  }{}\@ifundefined{ifGPblacktext}{\newif\ifGPblacktext
    \GPblacktexttrue
  }{}\let\gplgaddtomacro\g@addto@macro
\gdef\gplbacktext{}\gdef\gplfronttext{}\makeatother
  \ifGPblacktext
\def\colorrgb#1{}\def\colorgray#1{}\else
\ifGPcolor
      \def\colorrgb#1{\color[rgb]{#1}}\def\colorgray#1{\color[gray]{#1}}\expandafter\def\csname LTw\endcsname{\color{white}}\expandafter\def\csname LTb\endcsname{\color{black}}\expandafter\def\csname LTa\endcsname{\color{black}}\expandafter\def\csname LT0\endcsname{\color[rgb]{1,0,0}}\expandafter\def\csname LT1\endcsname{\color[rgb]{0,1,0}}\expandafter\def\csname LT2\endcsname{\color[rgb]{0,0,1}}\expandafter\def\csname LT3\endcsname{\color[rgb]{1,0,1}}\expandafter\def\csname LT4\endcsname{\color[rgb]{0,1,1}}\expandafter\def\csname LT5\endcsname{\color[rgb]{1,1,0}}\expandafter\def\csname LT6\endcsname{\color[rgb]{0,0,0}}\expandafter\def\csname LT7\endcsname{\color[rgb]{1,0.3,0}}\expandafter\def\csname LT8\endcsname{\color[rgb]{0.5,0.5,0.5}}\else
\def\colorrgb#1{\color{black}}\def\colorgray#1{\color[gray]{#1}}\expandafter\def\csname LTw\endcsname{\color{white}}\expandafter\def\csname LTb\endcsname{\color{black}}\expandafter\def\csname LTa\endcsname{\color{black}}\expandafter\def\csname LT0\endcsname{\color{black}}\expandafter\def\csname LT1\endcsname{\color{black}}\expandafter\def\csname LT2\endcsname{\color{black}}\expandafter\def\csname LT3\endcsname{\color{black}}\expandafter\def\csname LT4\endcsname{\color{black}}\expandafter\def\csname LT5\endcsname{\color{black}}\expandafter\def\csname LT6\endcsname{\color{black}}\expandafter\def\csname LT7\endcsname{\color{black}}\expandafter\def\csname LT8\endcsname{\color{black}}\fi
  \fi
    \setlength{\unitlength}{0.0500bp}\ifx\gptboxheight\undefined \newlength{\gptboxheight}\newlength{\gptboxwidth}\newsavebox{\gptboxtext}\fi \setlength{\fboxrule}{0.5pt}\setlength{\fboxsep}{1pt}\definecolor{tbcol}{rgb}{1,1,1}\begin{picture}(7920.00,3960.00)\gplgaddtomacro\gplbacktext{\colorrgb{0.00,0.00,0.00}\put(251,551){\makebox(0,0)[r]{\strut{}\scriptsize $10^{-4}$}}\colorrgb{0.00,0.00,0.00}\put(251,1338){\makebox(0,0)[r]{\strut{}\scriptsize $10^{-3}$}}\colorrgb{0.00,0.00,0.00}\put(251,2125){\makebox(0,0)[r]{\strut{}\scriptsize $10^{-2}$}}\colorrgb{0.00,0.00,0.00}\put(251,2913){\makebox(0,0)[r]{\strut{}\scriptsize $10^{-1}$}}\colorrgb{0.00,0.00,0.00}\put(251,3700){\makebox(0,0)[r]{\strut{}\scriptsize $10^{0}$}}\colorrgb{0.00,0.00,0.00}\put(352,383){\makebox(0,0){\strut{}\scriptsize $10^{2}$}}\colorrgb{0.00,0.00,0.00}\put(1940,383){\makebox(0,0){\strut{}\scriptsize $10^{4}$}}\colorrgb{0.00,0.00,0.00}\put(3529,383){\makebox(0,0){\strut{}\scriptsize $10^{6}$}}\csname LTb\endcsname \put(3546,3578){\makebox(0,0)[l]{\strut{}\scriptsize 1}}\csname LTb\endcsname \put(4356,3102){\makebox(0,0)[l]{\strut{}\scriptsize $\tfrac 12$}}\csname LTb\endcsname \put(1249,837){\makebox(0,0){\strut{}\scriptsize 1}}\csname LTb\endcsname \put(558,1595){\makebox(0,0)[r]{\strut{}\scriptsize 1}}\csname LTb\endcsname \put(7188,1063){\makebox(0,0)[l]{\strut{}\scriptsize $\alpha = 2$}}}\gplgaddtomacro\gplfronttext{\csname LTb\endcsname \put(5382,1363){\makebox(0,0)[l]{\strut{}\scriptsize\parbox{85pt}{$(|J|, \|f-A_{\text{ker}}^{\bm X_J} f\|_{L_2})$}$\Big(\frac{|J|}{\log|J|}\Big)^{-1.02}$}}\csname LTb\endcsname \put(5382,1722){\makebox(0,0)[l]{\strut{}\scriptsize\parbox{85pt}{$(|J|, \|f-S_{\mathcal B}^{\bm X_J} f\|_{L_2})$}$\Big(\frac{|J|}{\log|J|}\Big)^{-1.07}$}}\csname LTb\endcsname \put(5382,2082){\makebox(0,0)[l]{\strut{}\scriptsize\parbox{85pt}{$(|J|, \sqrt[4]{\mathcal S_n(\bm z)})$}$\Big(\frac{|J|}{\log|J|}\Big)^{-0.76}$}}\csname LTb\endcsname \put(5382,2441){\makebox(0,0)[l]{\strut{}\scriptsize\parbox{85pt}{$(n, \|f-A_{\text{ker}}^{\bm X} f\|_{L_2})$}$n^{-0.89}$}}\csname LTb\endcsname \put(5382,2801){\makebox(0,0)[l]{\strut{}\scriptsize\parbox{85pt}{$(n, \|f-S_{\mathcal B}^{\bm X} f\|_{L_2})$}$n^{-0.55}$}}\csname LTb\endcsname \put(5382,3160){\makebox(0,0)[l]{\strut{}\scriptsize\parbox{85pt}{$(n, \|f-A_{\mathcal A}^{\bm X} f\|_{L_2})$}$n^{-0.51}$}}\csname LTb\endcsname \put(5382,3520){\makebox(0,0)[l]{\strut{}\scriptsize\parbox{85pt}{$(n, \sqrt[4]{\mathcal S_n(\bm z)})$}$n^{-0.39}$}}\csname LTb\endcsname \put(2615,119){\makebox(0,0){\strut{}\scriptsize number of points}}\csname LTb\endcsname \put(2615,3939){\makebox(0,0){\strut{}reciprocal function $d=100$, $q=2.5$}}}\gplbacktext
    \put(0,0){\includegraphics[width={396.00bp},height={198.00bp}]{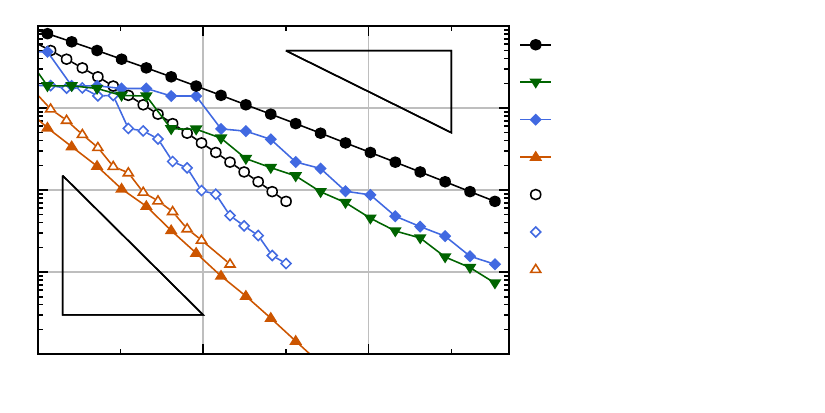}}\gplfronttext
  \end{picture}\endgroup
     }
    \caption{$L_2$-error of the approximations of the reciprocal function in $d=100$ dimensions and with decay parameter $q=2.5$ and smoothness parameter $\alpha=2$ with respect to the number of sampling points for the classical lattice algorithm, the least squares method, and the kernel method using the full and subsampled lattice.}\label{fig:reciprocal2.5_alpha2_E}
\end{figure} Since we deal with an analytic function, these results are not representative for the worst-case error and therefore the doubling of the rate \cite{SK25} becomes more apparent.
The high dimension $d=100$ diminishes the rates we expected with respect to $\alpha$.
For $\alpha=1$ (top) the kernel method with both the full as well as for the subsampled lattice perform a lot better than the corresponding least squares approximation, with empirical rates of $1.41$ and $1.89$ versus $0.84$ and $1.65$, respectively.
For $\alpha=2$ (bottom), kernel method and least squares approximation with subsampled lattice perform equally well, with empirical rates of $2.79$ and $2.8$, respectively.

In \Cref{fig:reciprocal6_E} we plot the corresponding results with the relatively slow decay parameter $q=2.5$ and $\alpha=2$.
This problem is considerably harder.
The empirical rates drop significantly to be just above 1 for both the kernel method and least squares approximation with the subsampled lattice.

The frequencies used in the $q=6$ example in \Cref{fig:reciprocal6_E} by the least squares approximation were nonzero only in the first $11$ and $23$ components for $\alpha=1$ and $\alpha=2$, respectively, i.e., the least squares method treated the remaining dimensions as irrelevant, while still achieving errors of $10^{-6}$ and $10^{-7}$.
For the $q=2.5$ case in \Cref{fig:reciprocal2.5_alpha2_E} frequencies with nonzero components were encountered only in the first $52$ dimensions.
This indicates that the ``effective dimension'' is low for the $q=6$ example compared to the $q=2.5$ example.
In this case all the methods for the $q=2.5$ example behave a lot more like $(n, \mathcal S_n(\bm z))$ for the full lattice and $(|J|, \mathcal S_n(\bm z))$ for the subsampled lattice, in comparison to the $q=6$ example.
It is also an example in which the subsampled least squares error shows a slightly faster empirical rate of $1.07$ than the subsampled kernel method empirical rate $1.02$.

\subsection{Computational complexity} 

Since we are interested in the computational complexity, we measure the computation time and counted the number of iterations for the presented methods for approximating the reciprocal function with $q=2.5$ and $\alpha=2$.
Both the \texttt{LSQR} method for the least squares approximation and the \texttt{CG} method for the kernel method were equipped with a stopping criterion of the relative error being smaller than $10^{-8}$.
The iteration count of the subsampled methods is depicted in \Cref{fig:reciprocal2.5_alpha2_I}.
\begin{figure} \centering
    \begingroup
  \makeatletter
  \providecommand\color[2][]{\GenericError{(gnuplot) \space\space\space\@spaces}{Package color not loaded in conjunction with
      terminal option `colourtext'}{See the gnuplot documentation for explanation.}{Either use 'blacktext' in gnuplot or load the package
      color.sty in LaTeX.}\renewcommand\color[2][]{}}\providecommand\includegraphics[2][]{\GenericError{(gnuplot) \space\space\space\@spaces}{Package graphicx or graphics not loaded}{See the gnuplot documentation for explanation.}{The gnuplot epslatex terminal needs graphicx.sty or graphics.sty.}\renewcommand\includegraphics[2][]{}}\providecommand\rotatebox[2]{#2}\@ifundefined{ifGPcolor}{\newif\ifGPcolor
    \GPcolortrue
  }{}\@ifundefined{ifGPblacktext}{\newif\ifGPblacktext
    \GPblacktexttrue
  }{}\let\gplgaddtomacro\g@addto@macro
\gdef\gplbacktext{}\gdef\gplfronttext{}\makeatother
  \ifGPblacktext
\def\colorrgb#1{}\def\colorgray#1{}\else
\ifGPcolor
      \def\colorrgb#1{\color[rgb]{#1}}\def\colorgray#1{\color[gray]{#1}}\expandafter\def\csname LTw\endcsname{\color{white}}\expandafter\def\csname LTb\endcsname{\color{black}}\expandafter\def\csname LTa\endcsname{\color{black}}\expandafter\def\csname LT0\endcsname{\color[rgb]{1,0,0}}\expandafter\def\csname LT1\endcsname{\color[rgb]{0,1,0}}\expandafter\def\csname LT2\endcsname{\color[rgb]{0,0,1}}\expandafter\def\csname LT3\endcsname{\color[rgb]{1,0,1}}\expandafter\def\csname LT4\endcsname{\color[rgb]{0,1,1}}\expandafter\def\csname LT5\endcsname{\color[rgb]{1,1,0}}\expandafter\def\csname LT6\endcsname{\color[rgb]{0,0,0}}\expandafter\def\csname LT7\endcsname{\color[rgb]{1,0.3,0}}\expandafter\def\csname LT8\endcsname{\color[rgb]{0.5,0.5,0.5}}\else
\def\colorrgb#1{\color{black}}\def\colorgray#1{\color[gray]{#1}}\expandafter\def\csname LTw\endcsname{\color{white}}\expandafter\def\csname LTb\endcsname{\color{black}}\expandafter\def\csname LTa\endcsname{\color{black}}\expandafter\def\csname LT0\endcsname{\color{black}}\expandafter\def\csname LT1\endcsname{\color{black}}\expandafter\def\csname LT2\endcsname{\color{black}}\expandafter\def\csname LT3\endcsname{\color{black}}\expandafter\def\csname LT4\endcsname{\color{black}}\expandafter\def\csname LT5\endcsname{\color{black}}\expandafter\def\csname LT6\endcsname{\color{black}}\expandafter\def\csname LT7\endcsname{\color{black}}\expandafter\def\csname LT8\endcsname{\color{black}}\fi
  \fi
    \setlength{\unitlength}{0.0500bp}\ifx\gptboxheight\undefined \newlength{\gptboxheight}\newlength{\gptboxwidth}\newsavebox{\gptboxtext}\fi \setlength{\fboxrule}{0.5pt}\setlength{\fboxsep}{1pt}\definecolor{tbcol}{rgb}{1,1,1}\begin{picture}(7920.00,3100.00)\gplgaddtomacro\gplbacktext{\colorrgb{0.00,0.00,0.00}\put(2596,551){\makebox(0,0)[r]{\strut{}\scriptsize $10^{0}$}}\colorrgb{0.00,0.00,0.00}\put(2596,1386){\makebox(0,0)[r]{\strut{}\scriptsize $10^{1}$}}\colorrgb{0.00,0.00,0.00}\put(2596,2221){\makebox(0,0)[r]{\strut{}\scriptsize $10^{2}$}}\colorrgb{0.00,0.00,0.00}\put(2596,3055){\makebox(0,0)[r]{\strut{}\scriptsize $10^{3}$}}\colorrgb{0.00,0.00,0.00}\put(2697,311){\makebox(0,0){\strut{}\scriptsize $10^{2}$}}\colorrgb{0.00,0.00,0.00}\put(3521,311){\makebox(0,0){\strut{}\scriptsize $10^{3}$}}\colorrgb{0.00,0.00,0.00}\put(4344,311){\makebox(0,0){\strut{}\scriptsize $10^{4}$}}\colorrgb{0.00,0.00,0.00}\put(5168,311){\makebox(0,0){\strut{}\scriptsize $10^{5}$}}}\gplgaddtomacro\gplfronttext{\csname LTb\endcsname \put(5684,2876){\makebox(0,0)[l]{\strut{}$S_{\mathcal B}^{\bm X_J}f$}}\csname LTb\endcsname \put(5684,2516){\makebox(0,0)[l]{\strut{}$A_{\text{ker}}^{\bm X_J}f$}}\csname LTb\endcsname \put(2149,1803){\rotatebox{-270.00}{\makebox(0,0){\strut{}\scriptsize number of iterations}}}\csname LTb\endcsname \put(3949,119){\makebox(0,0){\strut{}\scriptsize $|J|$}}}\gplbacktext
    \put(0,0){\includegraphics[width={396.00bp},height={155.00bp}]{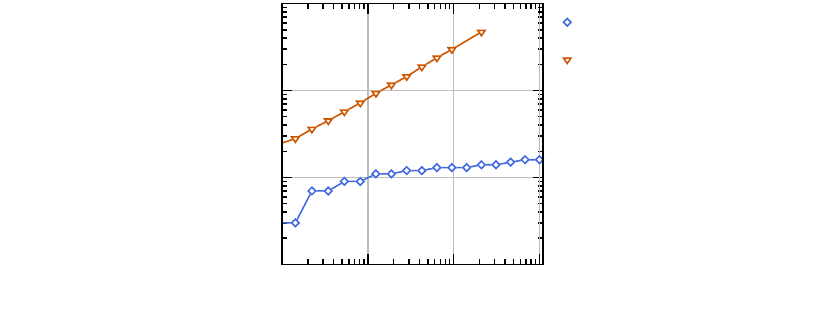}}\gplfronttext
  \end{picture}\endgroup
     \caption{Number of iterations with respect to the number of points for subsampled least squares approximation and kernel method.}\label{fig:reciprocal2.5_alpha2_I}
\end{figure} While the least squares approximation never exceeded $16$ iterations, the kernel method used up to $471$ iterations, growing like $|J|^{0.57}$.
This affirms the predicted cost with respect to the number of iterations in \Cref{tab:cost}.

For the actual cost, including precomputation and the matrix-vector products, we measure the elapsed time.
For the subsampled lattice we used the naive matrix-vector implementation as well as the FFT implementation discussed in \Cref{sec:implementation}.
Further, we modeled two scenarios: function evaluations to have no cost to evaluate and second, needing $0.1$ seconds to evaluate.
The latter resembles practical examples with function evaluations corresponding to PDE solutions or similar.
The results are depicted in \Cref{fig:reciprocal2.5_alpha2_T}.

\begin{figure} \centering
    \scalebox{0.9}{
    \begingroup
  \makeatletter
  \providecommand\color[2][]{\GenericError{(gnuplot) \space\space\space\@spaces}{Package color not loaded in conjunction with
      terminal option `colourtext'}{See the gnuplot documentation for explanation.}{Either use 'blacktext' in gnuplot or load the package
      color.sty in LaTeX.}\renewcommand\color[2][]{}}\providecommand\includegraphics[2][]{\GenericError{(gnuplot) \space\space\space\@spaces}{Package graphicx or graphics not loaded}{See the gnuplot documentation for explanation.}{The gnuplot epslatex terminal needs graphicx.sty or graphics.sty.}\renewcommand\includegraphics[2][]{}}\providecommand\rotatebox[2]{#2}\@ifundefined{ifGPcolor}{\newif\ifGPcolor
    \GPcolortrue
  }{}\@ifundefined{ifGPblacktext}{\newif\ifGPblacktext
    \GPblacktexttrue
  }{}\let\gplgaddtomacro\g@addto@macro
\gdef\gplbacktext{}\gdef\gplfronttext{}\makeatother
  \ifGPblacktext
\def\colorrgb#1{}\def\colorgray#1{}\else
\ifGPcolor
      \def\colorrgb#1{\color[rgb]{#1}}\def\colorgray#1{\color[gray]{#1}}\expandafter\def\csname LTw\endcsname{\color{white}}\expandafter\def\csname LTb\endcsname{\color{black}}\expandafter\def\csname LTa\endcsname{\color{black}}\expandafter\def\csname LT0\endcsname{\color[rgb]{1,0,0}}\expandafter\def\csname LT1\endcsname{\color[rgb]{0,1,0}}\expandafter\def\csname LT2\endcsname{\color[rgb]{0,0,1}}\expandafter\def\csname LT3\endcsname{\color[rgb]{1,0,1}}\expandafter\def\csname LT4\endcsname{\color[rgb]{0,1,1}}\expandafter\def\csname LT5\endcsname{\color[rgb]{1,1,0}}\expandafter\def\csname LT6\endcsname{\color[rgb]{0,0,0}}\expandafter\def\csname LT7\endcsname{\color[rgb]{1,0.3,0}}\expandafter\def\csname LT8\endcsname{\color[rgb]{0.5,0.5,0.5}}\else
\def\colorrgb#1{\color{black}}\def\colorgray#1{\color[gray]{#1}}\expandafter\def\csname LTw\endcsname{\color{white}}\expandafter\def\csname LTb\endcsname{\color{black}}\expandafter\def\csname LTa\endcsname{\color{black}}\expandafter\def\csname LT0\endcsname{\color{black}}\expandafter\def\csname LT1\endcsname{\color{black}}\expandafter\def\csname LT2\endcsname{\color{black}}\expandafter\def\csname LT3\endcsname{\color{black}}\expandafter\def\csname LT4\endcsname{\color{black}}\expandafter\def\csname LT5\endcsname{\color{black}}\expandafter\def\csname LT6\endcsname{\color{black}}\expandafter\def\csname LT7\endcsname{\color{black}}\expandafter\def\csname LT8\endcsname{\color{black}}\fi
  \fi
    \setlength{\unitlength}{0.0500bp}\ifx\gptboxheight\undefined \newlength{\gptboxheight}\newlength{\gptboxwidth}\newsavebox{\gptboxtext}\fi \setlength{\fboxrule}{0.5pt}\setlength{\fboxsep}{1pt}\definecolor{tbcol}{rgb}{1,1,1}\begin{picture}(7920.00,3100.00)\gplgaddtomacro\gplbacktext{\colorrgb{0.00,0.00,0.00}\put(692,616){\makebox(0,0)[r]{\strut{}\scriptsize $10^{-4}$}}\colorrgb{0.00,0.00,0.00}\put(692,1224){\makebox(0,0)[r]{\strut{}\scriptsize $10^{-3}$}}\colorrgb{0.00,0.00,0.00}\put(692,1832){\makebox(0,0)[r]{\strut{}\scriptsize $10^{-2}$}}\colorrgb{0.00,0.00,0.00}\put(692,2440){\makebox(0,0)[r]{\strut{}\scriptsize $10^{-1}$}}\colorrgb{0.00,0.00,0.00}\put(692,3049){\makebox(0,0)[r]{\strut{}\scriptsize $10^{0}$}}\colorrgb{0.00,0.00,0.00}\put(790,480){\makebox(0,0){\strut{}\scriptsize $10^{-2}$}}\colorrgb{0.00,0.00,0.00}\put(1801,480){\makebox(0,0){\strut{}\scriptsize $10^{0}$}}\colorrgb{0.00,0.00,0.00}\put(2812,480){\makebox(0,0){\strut{}\scriptsize $10^{2}$}}}\gplgaddtomacro\gplfronttext{\csname LTb\endcsname \put(179,1832){\rotatebox{-270.00}{\makebox(0,0){\strut{}\scriptsize $L_2$-error}}}\csname LTb\endcsname \put(2053,267){\makebox(0,0){\strut{}\scriptsize time in seconds}}}\gplgaddtomacro\gplbacktext{\colorrgb{0.00,0.00,0.00}\put(3852,616){\makebox(0,0)[r]{\strut{}\scriptsize $10^{-4}$}}\colorrgb{0.00,0.00,0.00}\put(3852,1224){\makebox(0,0)[r]{\strut{}\scriptsize $10^{-3}$}}\colorrgb{0.00,0.00,0.00}\put(3852,1832){\makebox(0,0)[r]{\strut{}\scriptsize $10^{-2}$}}\colorrgb{0.00,0.00,0.00}\put(3852,2440){\makebox(0,0)[r]{\strut{}\scriptsize $10^{-1}$}}\colorrgb{0.00,0.00,0.00}\put(3852,3049){\makebox(0,0)[r]{\strut{}\scriptsize $10^{0}$}}\colorrgb{0.00,0.00,0.00}\put(3950,480){\makebox(0,0){\strut{}\scriptsize $10^{0}$}}\colorrgb{0.00,0.00,0.00}\put(4961,480){\makebox(0,0){\strut{}\scriptsize $10^{2}$}}\colorrgb{0.00,0.00,0.00}\put(5972,480){\makebox(0,0){\strut{}\scriptsize $10^{4}$}}}\gplgaddtomacro\gplfronttext{\csname LTb\endcsname \put(7018,2904){\makebox(0,0)[l]{\strut{}\scriptsize$S_{\mathcal B}^{\bm X}f$}}\csname LTb\endcsname \put(7018,2613){\makebox(0,0)[l]{\strut{}\scriptsize$A_{\text{ker}}^{\bm X}f$}}\csname LTb\endcsname \put(7018,2323){\makebox(0,0)[l]{\strut{}\scriptsize$S_{\mathcal B}^{\bm X_J}f$ (LFFT)}}\csname LTb\endcsname \put(7018,2033){\makebox(0,0)[l]{\strut{}\scriptsize$S_{\mathcal B}^{\bm X_J}f$ (naive)}}\csname LTb\endcsname \put(7018,1743){\makebox(0,0)[l]{\strut{}\scriptsize$A_{\text{ker}}^{\bm X_J}f$ (FFT)}}\csname LTb\endcsname \put(7018,1453){\makebox(0,0)[l]{\strut{}\scriptsize$A_{\text{ker}}^{\bm X_J}f$ (naive)}}\csname LTb\endcsname \put(5213,267){\makebox(0,0){\strut{}\scriptsize time in seconds}}}\gplbacktext
    \put(0,0){\includegraphics[width={396.00bp},height={155.00bp}]{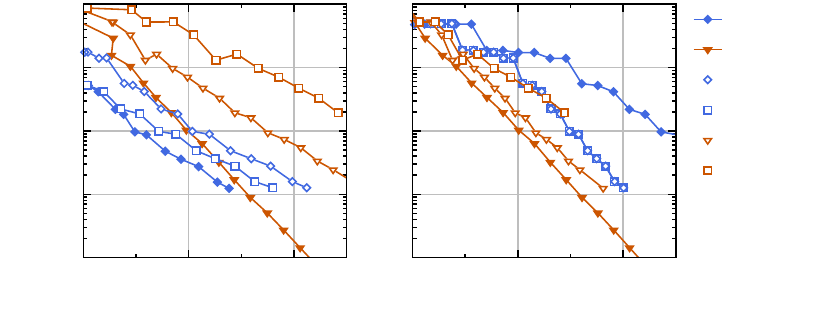}}\gplfronttext
  \end{picture}\endgroup
     }
    \caption{$L_2$-error of the numerical experiment with respect to the computation time.
    Left: function evaluation costs are not accounted for. Right: cost of single function evaluation is $0.1$ seconds.
    }\label{fig:reciprocal2.5_alpha2_T}
\end{figure} 

Without considering the cost for function evaluations the kernel method with the full lattice is the fastest when the number of points is large enough.
This is again due to making use of the higher smoothness of the function for a faster rate of decay.
The cost for all subsampled methods differ only in multiplicative constants for the error regardless of using the \texttt{FFT} or naive matrix-vector products.
The high overhead of the kernel method is due to the cost of evaluating the kernel and far more iterations needed for iterative solver, since the kernel matrix is not as well-conditioned as the least squares matrix.
In terms of memory, we stored the subsampled matrices for the naive implementation, which required us to stop at an earlier number of points as we werer limited to use $512$ Gigabytes of memory.
This could be overcome by computing the matrix on the fly but would add a larger computational overload.

With the addition of a small cost of $0.1$\,seconds for function evaluations, the graph resembles more the sampling complexity cost in \Cref{fig:reciprocal2.5_alpha2_E}.
Thus, in the right-hand setting when function evaluations are costly, the overhead of using an iterative solver with the the subsampling techniques of this paper is justified.
If the function in question is smooth, the kernel method excels due to the doubling of the rate effect \cite{SK25}.
 \section{Concluding remarks}\label{sec:conclusion} 

In this paper we used random subsampling to regain the optimal polynomial order of convergence for lattices while still maintaining the structure to apply fast Fourier algorithms.
The focus is on having a small initial lattice size, which determines the computational complexity of the approximation algorithms. We used the least squares approximation and the kernel method.
For Korobov spaces we achieve the optimal polynomial order of convergence 
$|J|^{-\alpha+\varepsilon}$ for $0<\varepsilon<\alpha$, with the initial lattice size $n\lesssim |J|^{2/\sqrt{1-\varepsilon/\alpha}}$, which is arbitrarily close to quadratic in terms of the actual  subsampled point set size $|J|$.
This improves on earlier results in \cite{BKPU22}, where there is a polynomial gap for the initial lattice size.

Similar to \cite{BKPU22} one could use further subsampling techniques from \cite{BSU23} or even the non-constructive techniques from \cite{DKU23} in order to improve on the logarithmic rate in the error bound.
Further, in the numerics the constants for the oversampling were ignored while still achieving the theoretical error rates.
This indicates possible improvement for the constants.
 
\vspace{10pt}

\noindent\textbf{Acknowledgement.}
T authors thank Dirk~Nuyens and Daniel~Potts for insightful discussions and valuable suggestions during the preparation of this work. We also thank both anonymous referees for their careful analysis and support. 
Further, w{}e acknowledge the financial support from the Australian Research Council Discovery Project DP240100769.
This research includes computations using the computational cluster Katana supported by Research Technology Services at UNSW Sydney.

\printbibliography

\end{document}